\documentclass[reqno,11pt]{amsart}
\usepackage{amsmath, latexsym, amsfonts, amssymb, amsthm, amscd}
\usepackage{graphics,epsf,psfrag}
\usepackage{graphics,epsf,psfrag}
\usepackage{stmaryrd}
\usepackage{xcolor}
\newcommand{\ext}{\mathrm{ext}}

\renewcommand{\bar}{\overline}

\newcommand{\lint}{\llbracket}
\newcommand{\rint}{\rrbracket}

\numberwithin{equation}{section}

\newtheorem{theorema}{Theorem}

\newcommand{\blue}{}
\newtheorem{theorem}{Theorem}[section]
\newtheorem{lemma}[theorem]{Lemma}
\newtheorem{proposition}[theorem]{Proposition}

\newtheorem{rem}[theorem]{Remark}

\newcommand{\ind}{\mathbf{1}}
\newcommand{\restrict}{\!\!\upharpoonright}
\newcommand{\Supp}{\mathrm{Supp}}
\newcommand{\larg}{\mathrm{large}}

\renewcommand{\tilde}{\widetilde}
\renewcommand{\hat}{\widehat}
\newcommand{\cc}{\complement}

\newcommand{\cA}{{\ensuremath{\mathcal A}} }
\newcommand{\cB}{{\ensuremath{\mathcal B}} }

\newcommand{\cP}{{\ensuremath{\mathcal P}} }

\newcommand{\cH}{{\ensuremath{\mathcal H}} }
\newcommand{\cC}{{\ensuremath{\mathcal C}} }

\newcommand{\cL}{{\ensuremath{\mathcal L}} }

\newcommand{\cD}{{\ensuremath{\mathcal D}} }

\newcommand{\cZ}{{\ensuremath{\mathcal Z}} }
\newcommand{\cI}{{\ensuremath{\mathcal I}} }
\newcommand{\cK}{{\ensuremath{\mathcal K}} }
\newcommand{\cG}{{\ensuremath{\mathcal G}} }

\newcommand{\bP}{{\ensuremath{\mathbf P}} }
\newcommand{\bQ}{{\ensuremath{\mathbf Q}} }
\newcommand{\bE}{{\ensuremath{\mathbf E}} }

\newcommand{\bL}{{\ensuremath{\mathbf L}} }

\newcommand{\cR}{{\ensuremath{\mathcal R}} }
\newcommand{\Int}{\mathrm{Int}}
\DeclareMathSymbol{\leqslant}{\mathalpha}{AMSa}{"36} % nicer `smaller or equal'
\DeclareMathSymbol{\geqslant}{\mathalpha}{AMSa}{"3E} % nicer `larger or equal'
\DeclareMathSymbol{\eset}{\mathalpha}{AMSb}{"3F}     % nicer `emptyset'
\newcommand{\Var}{\mathrm{Var}}        % \sum-like symbol for union
\newcommand{\sumtwo}[2]{\sum_{\substack{#1 \\ #2}}} % sum with 2 lines
\newcommand{\bbE}{{\ensuremath{\mathbb E}} }

\newcommand{\bbH}{{\ensuremath{\mathbb H}} }

\newcommand{\bbN}{{\ensuremath{\mathbb N}} }

\newcommand{\bbP}{{\ensuremath{\mathbb P}} }

\newcommand{\bbR}{{\ensuremath{\mathbb R}} }

\newcommand{\bbZ}{{\ensuremath{\mathbb Z}} }

\newcommand{\gb}{\beta}
\newcommand{\gep}{\varepsilon}       % \ge already exists...

\newcommand{\gG}{\Gamma}

\newcommand{\gD}{\Delta}

\newcommand{\go}{\omega}

\newcommand{\gO}{\Omega}
\newcommand{\gl}{\lambda}
\newcommand{\gL}{\Lambda}

\makeatletter
\def\captionfont@{\footnotesize}
\def\captionheadfont@{\scshape}

\long\def\@makecaption#1#2{%
  \vspace{2mm}
  \setbox\@tempboxa\vbox{\color@setgroup
    \advance\hsize-6pc\noindent
    \captionfont@\captionheadfont@#1\@xp\@ifnotempty\@xp
        {\@cdr#2\@nil}{.\captionfont@\upshape\enspace#2}%
    \unskip\kern-6pc\par
    \global\setbox\@ne\lastbox\color@endgroup}%
  \ifhbox\@ne % the normal case
    \setbox\@ne\hbox{\unhbox\@ne\unskip\unskip\unpenalty\unkern}%
  \fi
  \ifdim\wd\@tempboxa=\z@ % this means caption will fit on one line
    \setbox\@ne\hbox to\columnwidth{\hss\kern-6pc\box\@ne\hss}%
  \else % tempboxa contained more than one line
    \setbox\@ne\vbox{\unvbox\@tempboxa\parskip\z@skip
        \noindent\unhbox\@ne\advance\hsize-6pc\par}%
\fi
  \ifnum\@tempcnta<64 % if the float IS a figure...
    \addvspace\abovecaptionskip
    \moveright 3pc\box\@ne
  \else % if the float IS NOT a figure...
    \moveright 3pc\box\@ne
    \nobreak
    \vskip\belowcaptionskip
  \fi
\relax
}
\makeatother
\def\writefig#1 #2 #3 {\rlap{\kern #1 truecm
\raise #2 truecm \hbox{#3}}}

\newcommand{\tf}{\textsc{f}}

\newcommand{\kC}{{\mathfrak{C}}}

\title[Disordered pinning for Solid-On-Solid]{Solid-On-Solid interfaces with disordered pinning}

\author{Hubert Lacoin}
\address{
  IMPA, Institudo de Matem\'atica Pura e Aplicada, Estrada Dona Castorina 110
Rio de Janeiro, CEP-22460-320, Brasil. 
}

\begin{document}

\begin{abstract}
We investigate the localization transition for a simple model of interface which interacts with an inhomonegeous defect plane. The interface is modeled by the graph of a function $\phi: \bbZ^2 \to \bbZ$ and the disorder is given by a fixed realization of a field of IID centered random variables $(\go_x)_{x\in \bbZ^2}$. The Hamiltonian of the system depends on three parameters $\gb,\alpha>0$ and $h\in \bbR$  which determine  respectively the intensity of nearest neighbor interaction the amplitude of disorder and  the mean value of the  interaction with the substrate. It is given by the expression
$$\cH(\phi):= \gb\sum_{x\sim y} |\phi(x)-\phi(y)|- \sum_{x}
(\alpha\go_x+h)\ind_{\{\phi(x)=0\}},$$ 
We focus on the large-$\beta$/rigid phase phase of the Solid-On-Solid (SOS) model. In that regime, we provide a sharp description of the phase transition in $h$ from a localized phase  to a delocalized one corresponding respectively to a positive and vanishing fraction of points 
with $\phi(x)=0$. We prove that the critical value for $h$ corresponds to that of the annealed model and is given by $h_c(\alpha)= -\log \bbE[e^{\alpha \go}]$, and that near the critical point, the free energy displays the following critical behavior
$$\bar \tf_\gb(\alpha,h_c+u )\stackrel{u\to 0+}{\sim} \max_{n\ge 1} \left\{\theta_1 e^{-4\gb n} u-  \frac{1}{2}\theta^2_1 e^{-8\gb n} \frac{\Var\left[e^{\alpha \go}\right]}{\bbE \left[ e^{\alpha \go} \right]^2}\right\}.$$
The positive constant $\theta_1(\gb)>0$ is defined by the asymptotic probability of spikes  {\blue under $\bP_{\beta}$}, the infinite volume SOS measure  with $0$ boundary condition, that is,  $\theta_1(\beta):=\lim_{n\to \infty}  e^{4\beta n}\bP_{\beta} (\phi({\bf 0})=n).$ This particular form of asymptotic behavior is the signature of an accumulation of layering transitions: At the level of heuristics, the solution $n_u$ to the variational problem appearing in the asymptotics corresponds to the typical distance to the defect plane at which $\phi$ localizes. The {\blue discontinuity points for  the derivative of the functional} appearing in the asymptotics correspond to {\blue the value of $u$ at}  which this integer observable {\blue undergoes a transition. These transitions accumulates at zero and $n_u$ diverges to infinity as $u\to 0+$.}
\\[10pt]
  \\[10pt]
\end{abstract}

\maketitle

\tableofcontents

\section{Introduction}

The Solid-On-Solid model (SOS) was introduced in \cite{cf:BCF, cf:T} as an effective interface model for lattice systems displaying phase coexistence, a prototypical example being the low temperature 3D-Ising model with mixed boundary condition \cite{cf:Dob}. It provides a simplified framework to study the behavior of interfaces while conserving a rich and interesting phenomenology.
We   refer the reader to the recent survey concerning effective interfaces \cite{cf:IV} and to the introduction of \cite{cf:part1, cf:HL} for an overview centered on the SOS model.
 
 \medskip

 The two dimensional SOS model is known to display a phase transition from a low temperature phase where interfaces are rigid to a high temperature one where the interfaces are rough \cite{cf:BM, cf:FS, cf:Swend}.
 The rigid phase is the best understood of the two : 
 The representation of interfaces in terms of level lines allows for the application of cluster expansion techniques \cite{cf:KP} which can yield detailed information on the model such as exponential decay of correlations \cite{cf:BM}.
 
 \medskip
 
The focus of the present paper is to study the interaction of an SOS interface in the rigid phase with a defect plane. On top of the gradient nearest-neighbor interaction corresponding to the SOS Hamiltonian, we add an energetic reward (or penalty) for points located at level $0$.
The defect plane of interaction can be crossed. In other words, we allow our interface to visit both the positive and negative half-space.

\medskip

In the homogeneous setup, when the intensity of the interaction with the substrate, call it $h$, is the same at every point, one can show  using rather soft arguments (see Proposition \ref{firstorder} below) that the system exhibits a first order phase transition: {\blue the free energy's derivative - which corresponds to the asymptotic contact fraction - displays a discontinuity at the critical point}. 

\medskip

The understanding of the disordered model is a more delicate matter and the main focus of the present paper.
We want to understand whether and how the transition is modified when disorder is introduced in the interaction field, that is, when $h$ is replaced by $h+\alpha \go_x$ where $(\go_x)$ is an IID centered field.
The understanding of the effect of heterogeneities in a statistical mechanics model is in general a challenging problem which has attracted the attention of theoretical  and mathematical physicists
(see e.g. \cite{cf:AW, cf:AHP} and references therein for the case of the random field Ising model).
In the specific case of pinning models, the study of disorder relevance for one dimensional interfaces has given rise to a plentiful literature partly motivated by a connection to the DNA denaturation phenomenon (see \cite{cf:GB, cf:G} for a review and references). The case of higher dimension was explored only more recently in the case of Lattice Free Field interfaces \cite{cf:CM1,cf:GL, cf:GL2, cf:GL3, cf:FF2} for which heights takes values in $\bbR$.

\medskip

Our aim in studying a discrete interface model such as the SOS case is to understand the combined influence of lattice effects and disorder  on the phase transition. 
Our main result (Theorem \ref{mainres}) gives a detailed picture of the phase transition :
We prove that the critical value for $h$ is the same as that of the annealed system (which for our model coincides with the homogeneous system) and we identify the asymptotic behavior of the free energy  close to the critical point.

\medskip

This asymptotic behavior differs from that of the homogeneous model, in various aspects.
A first major difference with the homogeneous model is that the free energy growth at criticality is quadratic in $(h-h_c)$ instead of linear in the homogeneous case. This change of power exponent (from $1$ to $2$) is analogous to what was observed for the Lattice Gaussian Free Field (LGFF) in dimension larger than $3$ \cite{cf:GL3}
(when $d\ge 3$ the variance of the Lattice Free Field is uniformly bounded which makes the model somehow similar to the SOS model in the rigid phase,
 the two dimensional LGFF displays a very different behavior, see \cite{cf:FF2} for details).
Also, and this is perhaps the most novel aspect of our result, we identify a phenomenon which is specific to the discrete nature of the  SOS model : The asymptotics of the free energy is not a pure power. It is a piecewise affine function {\blue whose derivative displays an accumulative sequence of
discontinuity point} (given in \eqref{optimize} below). These points correspond to changes of the typical localization height of the SOS field which can take only integer values.

\medskip

Our result thus indicates that besides the localization transition, the system could undergo countably many \textit{layering transitions}.
 This is only a heuristic reasoning and it is in general a difficult task to prove that the free energy  itself - and not only the asymptotic approximation given in \eqref{optimize} - displays {\blue discontinuities in its derivative}.
In the present paper, we content ourselves with the proof of free energy asymptotic and leave the (challenging) question of the presence of angular points on the free energy for future endeavors.
To our knowledge, this is the first time in the litterature (both in Mathematics and Theoretical Physics) that a layering phenomenon triggered by disorder is identified. 

\medskip

Let us mention also the free energy's asymptotic expression displays similarity with the one found in 
\cite[Theorem 2.1]{cf:part1} for the homogeneous wetting of SOS interfaces
with an homogeneous interaction at level zero and half-space restriction 
though the mechanism triggering the layering phenomenon in that case is of a very different nature (see \cite{cf:ADM, cf:AY, cf:BMF, cf:CM, cf:chal,cf:DM, cf:HL} for more details on the subject).

\section{Model and results}

\subsection{The SOS model}

Consider $\gL$ a finite subset of  $\bbZ^d$ (equipped with its usual lattice structure) and let $\partial \gL$ denote its external boundary 
$$\partial \gL:= \{ x \in \bbZ^d \setminus \gL \ : \ \exists y \in \gL, \ x \sim y \}.$$
Setting  $\gO_{\gL}:=\{ \phi : \gL \to \bbZ\}$ and fixing an integer parameter $n\in \bbZ$ we define the Hamiltonian for the Solid-On-Solid (SOS) model with boundary condition $n$ as,
\begin{equation}\label{defhamil}
\cH^{n}_\gL(\phi):= \frac{1}{2}\sumtwo{(x, y) \in \gL^2}{ x \sim y} |\phi(x)-\phi(y)|
+\sumtwo{x\in \gL, y\in \partial\gL}{x\sim y} |\phi(x)-n|.
\end{equation}
The factor $1/2$ in front of the first sum is present so that each pair $\{x,y\}$ of neighboring points has total weight one in the sum.
Given $\gb>0$, we define
the SOS measure with boundary condition $n$, $\bP^{n}_{\gL,\gb}$ on $\gO_{\gL}$ by
\begin{equation}\label{defSOS}
\bP^{n}_{\gL,\gb}(\phi):= \frac{1}{\mathcal Z_{\gL,\gb}}e^{-\gb\cH^{n}_\gL(\phi)} \quad \text{where} \quad
\mathcal Z_{\gL,\gb}:=\sum_{\phi\in \gO_{\gL}}  e^{-\gb\cH^n_\gL(\phi)}.
\end{equation}
Note that, by vertical translation invariance, $\mathcal Z_{\gL,\gb}$ does not depend on $n$.
For readability, we drop the superscript $n$ in the notation in the special case $n=0$.
Observe that if $\gL^{(1)}$ and $\gL^{(2)}$ are disjoint  we 
 have 
 \begin{equation}\label{hicup}
 \cH_{\gL^{(1)}\cup \gL^{(2)}}(\phi)\le  \cH_{\gL^{(1)}}(\phi)+\cH_{\gL^{(2)}}(\phi)
 \end{equation}
which yields immediately 
\begin{equation}\label{hicup2}
 \mathcal Z_{\gL^{(1)}\cup \gL^{(2)},\gb}\ge \cZ_{\gL^{(1)},\gb}\cZ_{\gL^{(2)},\gb}.
\end{equation}
This property implies (we refer to \ \cite[Exercise 3.3]{cf:Vel} for a proof  of this classical fact for the partition function of the Ising model) 
the existence of the following limit
\begin{equation}\label{revert}
\lim_{n \to \infty} \frac{1}{|\gL_n|}\log  \cZ_{\gL,\gb}=\sup_n \frac{1}{|\gL_n|}\log  \cZ_{\gL,\gb} =\tf(\gb),
\end{equation} 
along any sequence of rectangles $\gL_n:=\prod_{i=1}^d\lint a^{(i)}_n, b^{(i)}_n\rint$ which satisfies  $ \lim_{n\to \infty }\min_i(b^{(i)}_n-a^{(i)}_n)=\infty$ (here and in the rest of the paper, we use the notation 
$\lint a, b\rint =[a,b]\cap \bbZ$ for $a<b$).
To check that $\tf(\beta)$ is finite for $\beta>0$, the reader can check by a simple computation that $Z_{\gL,\beta}\le \left( \frac{e^{\beta}+1}{e^{\beta}-1}\right)^{|\gL|}$.

\begin{rem}
With additional efforts, using the specifics of the model, one can extend the statement and show that \eqref{revert} holds as soon as 
$\lim_{n\to \infty} |\partial \gL_n|/|\gL_n|=0\, ,$ but this is not required  for our analysis.
\end{rem}

When $\beta$ increases, the SOS model with constant boundary condition undergoes a phase transition from a rough phase where the variance of $\phi$ diverges  with the distance to the boundary $\partial \gL$ (see \cite{cf:FS}) to a rigid phase where the distribution of $\phi$ at any given site remains tight.
In the present paper we are solely interested in this low-temperature phase, so let us describe more accurately the results which are available for large $\beta$.
It is known  (cf.\ \cite[Theorem 2]{cf:BM})  that for $\beta \ge \gb_0$ sufficiently large,  $\bP_{\gL,\gb}$ converges (in the sense of finite dimensional marginals)
to an infinite volume measure $\bP_{\gb}$ defined on $\gO:=\{ \phi \ : \ \bbZ^2 \to \bbZ\}$. We introduce a quantitative version of the statement which requires 
the introduction of  some classic terminology.

\medskip

We say that a function $f: \ \gO_{\infty}:=(\bbZ)^{\bbZ^d}\to \bbR$ is local if there exists $x_1,\dots,x_k \in \bbZ^d$  and $\tilde f:  (\bbZ)^{k}\to \bbR$ 
such that $f(\phi)=\tilde f(\phi(x_1),\dots,\phi(x_k)).$ The minimal  choice (with respect to the inclusion) 
for the set of indices $\{x_1,\dots, x_k\}$ is called the support of $f$ (and is denoted by $\Supp(f)$).
With some abuse of notation, whenever $\gL$ contains the support of $f$,  we extend $f$ to $\gO_{\gL}$ in the obvious way.
An event is called local if its indicator function is a local function.
For $A$ and $B$ two finite subsets of $\bbZ^d$ we set 
\begin{equation}\label{disset}
d(A,B):=\min_{x\in A, y\in B} |x-y|,
\end{equation}
where $|\ \cdot \ |$ denote the $\ell_1$ distance.
The following result follows from the proof in \cite{cf:BM}
(see also \cite{cf:HL} for details).
\begin{theorema}\label{infinitevol}
There exists constants $\gb_0(d)>0$ and $c>0$ such that 
for any $\gb>\gb_0$, 
there exists a measure $\bP_{\gb}$ defined  on $\gO_{\infty}$ such that for every local function 
$f: \gO_{\infty}\to [0,1]$ with  $\Supp(f)=A$, and every $\gL$ that contains $A$
\begin{equation}\label{decayz}
| \bE_{\gL,\gb}[f(\phi)]-\bE_{\gb}[f(\phi)] |\le  |A| e^{-c \gb d(\partial \gL, A)}.
\end{equation}
\end{theorema}

\subsection{The SOS model with interaction at level $0$}

We introduce now a modification of the SOS measure by introducing an extra term in the Hamiltonian to model an interaction of the interface at level $0$.
Consider $(\go_x)_{x\in \bbZ^d}$ a fixed realization of an IID random field indexed by $\bbZ^d$ (we let $\bbP$ denote its distribution). We assume that our random variables 
have finite exponential moments of all orders
\begin{equation}\label{allorder}
\forall \alpha \in \bbR,\  \gl(\alpha):=\log \bbE[e^{\alpha \go_x}]<\infty.
\end{equation} 
We also assume  (without loss of generality) that $\go_x$ has zero mean.
Fixing $\alpha>0$ and $h\in \bbR$, we define $\bP^{n,h,\alpha,\go}_{\gL,\gb}$ to be a modified version of the SOS measure, where for each $x$ such that $\phi(x)=0$, an energy term $\alpha \go_x- \gl(\alpha)+h$ is added to the Hamiltonian.  
That is, setting 
$\delta_x:= \ind_{\{ \phi(x)=0\}},$
\begin{equation}
\bP^{n,h,\alpha,\go}_{\gL,\gb}(\phi):= \frac{1}{\cZ^{n,h,\alpha,\go}_{\gL,\gb}}e^{-\gb\cH^{n}_\gL(\phi)+ \sum_{x\in \gL}(\alpha \go_x- \gl(\alpha)+h)\delta_x}
\end{equation}
with 
\begin{equation}
\cZ^{n,h,\alpha,\go}_{\gL,\gb}:= \sum_{\phi \in \gO_{\gL}}e^{-\gb\cH^{n}_\gL(\phi)+ \sum_{x\in \gL}(\alpha \go_x- \gl(\alpha)+h)\delta_x}.
\end{equation}
The term $-\gl(\alpha)$ is present by mere convention (it just corresponds to a shift in $h$) but turns out to be practical when considering the annealed model (see in particular the inequality \eqref{annealed} below).

Setting $\gL_N:=\lint 1,N\rint^2$ we replace the subscript $\gL$ by $N$ in the notation when $\gL_N$ is considered. When $\alpha=0$, neither the partition function nor the probability measure introduced above depend on $\go$ and they are simply denoted as $\cZ^{n,h}_{N,\gb}$ and $\bP^{n,h}_{N,\gb}$ respectively. As stated before, we drop $n$ from the notation when considering $0$ boundary condition.

\medskip

Our aim is to understand the asymptotic properties of $\bP^{h,\alpha,\go}_{N,\gb}$ when $N$ tends to infinity.
In particular we want to understand the effect of disorder, that is, how 
$\bP^{h,\alpha,\go}_{N,\gb}$ differs from $\bP^{h}_{N,\gb}$ when $\alpha>0$.
To this purpose we study  the asymptotic free energy per unit of volume (we simply refer to it as the free energy) whose existence follows from combining now standard arguments which can be found e.g.\ in \cite{cf:CSY, cf:Vel, cf:G}. We include the details of the proof in Appendix \ref{appintro}.

\begin{proposition}\label{freeen}

For every $\alpha, \beta>0$ and $h\in \bbR$, the following limit exists and does not depend on $n$ {\blue nor on $\go$}. The convergence holds both in $L_1(\bbP)$ and in the almost sure sense 
\begin{equation}\label{convergence}
\lim_{N \to \infty}\frac{1}{N^d}\log  \cZ^{n,h,\alpha,\go}_{N,\gb} =:\tf(\gb,\alpha,h).
\end{equation} 
{\blue In particular we also have 
$$\lim_{N \to \infty}\frac{1}{N^d}\bbE \left[\log  \cZ^{n,h,\alpha,\go}_{N,\gb}\right]=\tf(\gb,\alpha,h).$$}
The function $h \mapsto \tf(\gb,\alpha,h)$ is convex non-decreasing and we have whenever the derivative exists 
\begin{equation}\label{contact}
\partial_h \tf(\gb,\alpha,h):= \lim_{N\to \infty} \frac{1}{N^d} \bE^{n,h,\alpha,\go}_{N,\gb}\left[\sum_{x\in \gL_N} \delta_x\right].
\end{equation}
Setting $\tf(\gb,h):= \tf(\beta,0,h)$ we have for every $\alpha>0$,
\begin{equation}\label{annealed}
 \tf(\beta,h-\gl(\alpha)) \le \tf(\beta,\alpha,h)\le  \tf(\beta,h).
\end{equation}
We also have for every $\alpha, \beta$ and $h$
\begin{equation}\label{labound}
 \tf(\beta,\alpha,h)\ge \tf(\gb).
 \end{equation}

\end{proposition}

\noindent We  let $\bar \tf_{\gb}(\alpha,h)$ denote the free energy difference produced by the interaction with the defect plane by setting 
$$ \bar \tf_{\gb}(\alpha,h):= \tf(\gb,\alpha,h)- \tf(\gb)$$
and simply write $\bar \tf_\gb(h)$ when $\alpha=0$.
We have for any fixed value of $n$ 
\begin{equation}\label{reduced}
  \bar \tf_{\gb}(\alpha,h)= \lim_{N\to \infty} \frac{1}{N^d} \log \bE^n_{N,\beta} \left[ e^{\sum_{x\in \gL_N}(\alpha \go_x- \gl(\alpha)+h)\delta_x} \right]=: \lim_{N\to \infty} \frac{1}{N^d} \log Z^{n,h,\alpha,\go}_{\beta,N}.
\end{equation}
We refer to  $Z^{n,h,\alpha,\go}_{\beta,N}= \frac{\cZ^{n,h,\alpha,\go}_{N,\beta}}{\cZ_{N,\beta}}$ as the reduced partition function.
Note that from \eqref{labound} $\bar \tf_{\gb}(\alpha,h)\ge 0$ for all $h\in \bbR$. Moreover it is quite immediate to check that $\bar \tf_{\gb}(h)\le 0$ when $h\le 0$  (and hence that 
$\bar \tf_{\gb}(h)= 0$ when $h\le 0$) . By \eqref{annealed} we also have $\bar \tf_{\gb}(\alpha,h)= 0$ when $h\le 0$.
Let us set 
$$h_c(\beta,\alpha):= \inf\{ h : \ \bar \tf_{\beta}(\alpha,h)>0  \}.$$ 
According to \eqref{contact}, $h_c(\beta,\alpha)$ separates two phases, with the asymptotic contact fraction vanishing
if $h<h_c(\beta,\alpha)$ and remains bounded away from zero when $h>h_c(\beta,\alpha)$ or more precisely
\begin{equation}
 \begin{split} \lim_{N\to \infty} \frac{1}{N^d} \bE^{n,h,\alpha,\go}_{N,\gb}\left[\sum_{x\in \gL_N} \delta_x\right]&=0 \quad \text{ if } h<h_c(\beta,\alpha),\\
  \liminf_{N\to \infty} \frac{1}{N^d} \bE^{n,h,\alpha,\go}_{N,\gb}\left[\sum_{x\in \gL_N} \delta_x\right]&>0 \quad \text{ if } h>h_c(\beta,\alpha).
 \end{split}
\end{equation}
{\blue Note that the asymptotic contact fraction is also zero at the critical point $h_c(\beta,\alpha)$ if  $\partial_h \tf_{\beta}(\alpha,h_c(\beta,\alpha))=0$. In the case when $\partial_h \tf_{\beta}(\alpha,h_c(\beta,\alpha))>0$ then \eqref{contact} does not provide information but one may conclude using other arguments see Remark \ref{zetruc} below.}
We focus on the large-$\beta$ regime in which interfaces are rigid (cf.\ Theorem \ref{infinitevol}).
In that case, the homogeneous model ($\alpha=0$) displays a first order phase transition at $h=0$, in the sense that asymptotic contact fraction displays a discontinuity. The statement is a rather direct consequence of rigidity of the interfaces,  the  proof is included  in Appendix \ref{appintro} for completeness.

\begin{proposition}\label{firstorder}
 When $\beta\ge \beta_0$ (given by Theorem \ref{infinitevol}) we have $h_c(\beta,0)=0$ and there exists a constant $c_{\beta}>0$.
 \begin{equation}
  \bar \tf_\gb(h)\stackrel{h\to 0+}{\sim} c_{\beta} h.
 \end{equation}
\end{proposition}

\begin{rem}
 The natural guess is that one should have $c_{\beta}=\bP_{\beta}(\phi({\bf 0})=0)$. It seems likely such a statement can be proved using cluster expansion techniques similar to the ones exposed e.g.\ in \cite{cf:HL} but this is out of the scope of the present paper.
\end{rem}

{\blue
\begin{rem}\label{zetruc} As a consequence of Theorem \ref{infinitevol} we have when $\beta=0$ and the fact that $h_c(\beta,0)=0$ we have
\begin{equation}
 \lim_{N\to \infty} \frac{1}{N^d} \bE^{n,h_c(\beta,0)}_{N,\beta}\left[\sum_{x\in \gL_N} \delta_x\right]=
 \bP_{\beta}(\phi({\bf 0})=0)>0.
 \end{equation}

\end{rem}}

From \eqref{annealed}, it follows that $h_c(\beta,\alpha)\ge 0$ for all $\alpha>0$. Hence a natural question is whether this inequality is strict. Another one is whether the order of the  phase transition is modified by the introduction of disorder. These questions are intimately related to that of disorder relevance: ``Does the introduction of a small amount of disorder in the system change the characteristics of the phase transition?''

\subsection{Presentation of the main result}

The main result presented in this paper aims at giving  a detailed  picture of the phase transition for the disordered system which goes beyond the identification of the critical points and  of the free energy  critical exponent.
For the sake of simplicity we restrict ourselves to the case of dimension $2$ for which the contour decomposition (see Section \ref{contour}) allows for more intuitive proofs. We are very confident  that the method extends to the case of higher dimension.

\medskip

We show first that when $\beta$ is sufficiently large (larger than $\beta_0$ given by Theorem \ref{infinitevol}), then $h_c(\beta,\alpha)=0$
for all $\alpha>0$.
We  also prove that the behavior at the vicinity of $h_c(\beta,\alpha)$ is different from the one observed in the homogeneous case.
A major difference is that $\bar \tf_{\beta}(\alpha,h)$ grows quadratically at the right of $0$ in the sense that 
 there exist two constants $c_{\alpha,\beta}$ and $C_{\alpha,\beta}>0$ such that for every $h\in [0,1]$
\begin{equation}\label{quadra}
 c h^2 \le \bar \tf_{\beta}(\alpha,h)\le C h^2.
\end{equation}
However, whereas \eqref{quadra} might suggest it at first glance, the quantity $\bar \tf_{\beta}(\alpha,h) h^{-2}$ does not converge when $h\to 0+$. The sharp asymptotic of  $\bar \tf_{\beta}(\alpha,h)$ is rather given by a function which is piecewise affine and whose angular points form a geometric sequence accumulating at $0$.
In order to give an precise formulation to this we need to 
 introduce the quantity $\theta_1$ (which depends on $\beta$)  
 which governs the probability of seing thin spikes appearing in the infinite volume SOS interface
\begin{equation}\label{singledouble}
\theta_1:= \lim_{n\to \infty} e^{4\gb n}  \bP_{\gb} \left[ \phi({\bf 0})\ge  n\right],
\end{equation}
For the existence and positivity of $\theta_1$ we refer to \cite[Proposition 4.5]{cf:part1} (see also \cite[Lemma 2.4]{cf:PLMST}). We can now state our main result.

\begin{theorem}\label{mainres}(d=2)
 There exists $\beta_0$ such that for $\beta\ge \beta_0$, we have for every $\alpha>0$, $h_c(\gb,\alpha)=0$.
 Furthermore we have 
 
  \begin{equation}\label{crooti}
  \bar \tf_\gb(\alpha,h)\stackrel{h\to 0+}{\sim} G_{\gb}(\alpha,h)
 \end{equation}
 where 
 \begin{equation}\label{optimize}
  G_{\gb}(\alpha,h):= \max_{n \ge 0} \left[ \theta_1 e^{-4\gb n} h -\frac{1}{2}\theta^2_1 e^{-8\gb n}{\blue \frac{\Var(e^{\alpha\go_{\bf 0}})}{\bbE\left[e^{\alpha\go_{\bf 0}}\right]^2}}\right].
 \end{equation}

\end{theorem}

\begin{figure}[ht]
\begin{center}
\leavevmode
\epsfysize = 8 cm
\psfrag{h}{$h$}
\psfrag{G(a,h)}{$G_{\beta}(\alpha,h)$}
\psfrag{0}{$0$}
\epsfbox{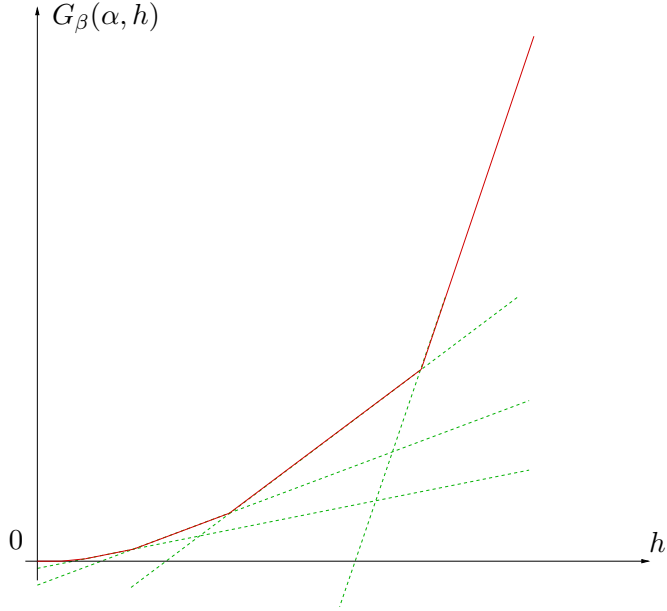}
\end{center}
\caption{\label{gah} {\blue
The asymptotic free energy equivalent $G_{\beta}(\alpha,h)$ as a function of $h$ (solid red curve) represented together with four of the affine functions appearing in the $\max$ in \eqref{optimize} (dotted lines). The derivative of $G_{\beta}(\alpha,h)$ is discontinuous along a geometric sequence with  ratio $e^{-4\beta}$, the slopes of the affine pieces of the curve follow a geometric sequence with ratio $e^{-8\beta}$.
To make the picture more readable we have chosen $\beta= (\log 2)/4$ even though our result requires $\beta$ to be large.}}
\end{figure}

{\blue Recalling \eqref{allorder} we have  $\frac{\Var(e^{\alpha\go_{\bf 0}})}{\bbE\left[e^{\alpha\go_{\bf 0}}\right]^2}=e^{\gl(2\alpha)-2\gl(\alpha)}-1$}.
The expression \eqref{optimize}  provides an acute picture of the localization strategy of $\phi$ in the near critical regime, which we choose to discuss in the following subsection. {\blue Note that a consequence of \eqref{crooti} is that     $\bar \tf_\gb(\alpha,\cdot)$ is differentiable at $0$ and  $\partial_h\bar \tf_\gb(\alpha,h)=0$. Thus the asymptotic contact fraction vanishes at criticality.}

\subsection{Interpretation of the result} 
The expression of $G_{\gb}(\alpha,h)$ can be explained by describing the localization strategy for the interface $\phi$.
Because of the exponential decay of correlation exposed in Theorem \ref{infinitevol}, for large $n$ under $\bE^n_{N,\beta}$, 
the process $(\delta_x)_{x\in \gL_N}$ looks like an IID collection of Bernoulli variables of parameter $\bP_{\beta}( \phi ({\bf 0})=n)$ which can be approximated by 
$p_n:= \theta_1 e^{-4\beta n}$.

\medskip

The partition function which is obtained by replacing $\bE^n_{N,\beta}$ by the IID Bernoulli distribution with parameter $p_n$ in the reduced partition function (recall \eqref{reduced}) is equal to
$$ \prod_{x\in \gL_N} (1+ p_n (e^{\alpha \go_x-\gl(\alpha)+h}-1)). $$
The corresponding free energy obtained by taking the logarithm, then the expectation and dividing by the volume of the box is simply  $\bbE \left[  \log ( 1+ p_n (e^{\alpha \go-\gl(\alpha)+h}-1) \right]$.
Using a Taylor expansion, {\blue there exists a constant $C>0$ such that   for all $h\in [0,1]$ and $n\ge 1$} 
\begin{equation*}
 \left|\bbE \left[  \log ( 1+ p_n (e^{\alpha \go-\gl(\alpha)+h}-1) \right]-
 p_n h -  \frac{1}{2} p^2 _n(e^{\gl(2\alpha)-2\gl(\alpha)}-1)\right| \le    C |h^2 p_n + h p^2_n + p^3_n|.  
\end{equation*}
Hence noticing that for the maximizing value of $n$, $p_n$ is of order $h$ we have (recall the definition of  $G_{\gb}(\alpha,h)$ in \eqref{optimize})
\begin{equation}
  \max_{n\ge 1}\bbE \left[  \log ( 1+ p_n (e^{\alpha \go_x-\gl(\alpha)+h}-1) \right]\stackrel{h\to 0+}{\sim}  G_{\gb}(\alpha,h).
\end{equation}
Hence a heuristic interpretation of the result \eqref{crooti} is that for small values of $h$, typically under $\bP^{h,\alpha,\go}_{N,\beta}$, in the bulk of the box, the distribution of $\phi$ looks like $\bP^n_{N,\beta}$ where $n=n_h$ is the value that maximizes \eqref{optimize}. The interface localizes around height $n_h$ because it is the best manner 
to have an optimal contact density $p_n$. The entropic cost of jumping from height $0$ to $n_h$ near the boundary of the box is of order $N^{d-1}$ and is  compensated by the gain in the bulk which scales like $N^d$.

\medskip

This phenomenon of localization around a typical height (and the associated behavior of the free energy) is reminiscent of the layering phenomenon observed for the SOS model in the presence of a solid substrate \cite{cf:ADM, cf:part1, cf:HL}. The difference here is that the layering phenomenon is not triggered by entropic repulsion but by the presence of disorder. To our knowledge, our result is the first reported case of disorder-induced layering.

\subsection{Comparison with results obtained for other models}
\label{ozermodel}
The study of disorder relevance for pinning models has been a very active field of study in the past two decades. The focus has first been put on the problem of $1$ dimensional pinning which corresponds to a random walk interacting with a defect line (we refer to \cite{cf:GB, cf:G} for an historical introduction, reviews and  references).
In this case, the contact set possesses a renewal structure which helps in the analysis. A series of work (cited in alphabetical order)  \cite{cf:Ken, cf:AZ, cf:BL, cf:DGLT, cf:GLT, cf:Trep} allowed for a full characterization of disorder relevance and in \cite{cf:GT} it was shown that the free energy transition is always smoother than quadratic. In particular {\blue for a large class of 
centered random walks on $\bbZ$  (which includes the nearest neighbor random walk and the SOS model in dimension $1$)} it was proved that there is a shift of the disordered critical point w.r.t.\ the annealed one \cite{cf:GLT}, in contrasts with our main result.

\medskip

More recently an extensive answer to the question of disorder relevance has been given for higher dimensional surface models, in the case where $\phi$ is the lattice Gaussian Free Field (GFF) on $\bbZ^d$ 
(and $\delta_x$ is replaced by $\ind_{\{\phi(x)\in [-1,1]\}}$) with $d\ge 2$ \cite{cf:CM1, cf:GL, cf:GL2, cf:GL3, cf:FF2}.
The case which offers most similarities with the low temperature SOS model is  that of dimension $d\ge 3$ (for which it is known that the variance of the field is bounded).
In that case it was shown in \cite{cf:GL, cf:GL3} that while the value of the critical  point  $h_c$ is not affected by the introduction of inhomogeneities, disorder smoothens the phase transition.  The homogeneous model displays a phase transition of first order (like the model studied in the present paper, cf. Proposition \ref{firstorder}), while the critical growth of the free energy is quadratic (also similar to what we observe here cf. \eqref{quadra}) when disorder is present.

\medskip

 However, unlike for the SOS model, the asymptotic of the free energy when $h$ tends to $0$ is given by a pure power (recall \eqref{optimize}).
More precisely, in the GFF case, the free energy asymptotics is given by \cite[Theorem 1.1]{cf:GL3}
$$ \max_{p \in  [0,1]} \left[ p h -\frac{1}{2}p^2 (e^{\gl(2\alpha)-2\gl(\alpha)}-1)\right]= \frac{h^2}{2(e^{\gl(2\alpha)-2\gl(\alpha)}-1) }. $$ 
{\blue The expression is very similar to that in \eqref{optimize}, the main difference being that in the in the optimization problem in the l.h.s.\ $p$ is allowed to assume any value in $[0,1]$ instead of being constrained to belong to the set  $\{ p_n \}_{n\ge 1}= \{ \theta_1 e^{-4\beta_n}\}_{n\ge 1}$. The reason for the similarity is that, for both the GFF and  the SOS model, if the boundary condition is set to a high value, then the contact set looks like a Bernoulli field. Hence in both cases at a small entropic cost, one can change the boundary coundition so that the distribution of $(\delta_x)_{x\in \gL_N}$ is ``close'' to that of a Bernoulli fields.
The difference is that for the GFF, the contact set can assume any specified density $p$ in an interval of the form $[0,p_0]$ because  the density of contact under the infinite volume measure varies continuously with the $\bbR$-valued boundary condition.}
For the SOS model on the contrary, the only values of $p$ for which $(\delta_x)_{x\in \gL_N}$ can emulate a Bernoulli field with density $p$ are given by $\bP_{\beta}( \phi ({\bf 0})=n) \sim \theta_1 e^{-4\beta n}$,
hence the optimization procedure is performed only along that sequence.

\subsection{Presence/Absence of layering transitions}
The form taken by the free energy asymptotic is very reminiscent of 
that found for the (non disordered) wetting problem for the SOS model (see \cite[Theorem 2.1]{cf:part1}). Moreover 
it was shown in \cite{cf:HL} (see also \cite{cf:CM} and references therein for earlier similar results) that each {\blue discontinuity point in the derivative of the}  asymptotic approximation is also present on the free energy curve. This sequence of points of non-differentiability  corresponds to a  (countably infinite) sequence of first order phase transitions  which correspond to transitions of the typical height assumed by $\phi$ from one integer value to another.

\medskip

It is thus a natural question to ask whether such a sequence of first order phase transitions is also observed for the disordered SOS model. We leave this question open for future research.

\subsection{Organization of the paper}

In Section \ref{prelim}, we introduce some tools and estimates for the low temperature SOS measure, that are necessary for the proof of our main result. Most results have appeared in other references. The proof of complementary results  follow very similar ideas and are included Appendix \ref{lazt} for the sake of completeness.

\medskip

In Section \ref{sec:dalower}, we prove a quantitative version of the lower bound on the free energy displayed in Theorem \ref{mainres}.
This proof relies on using a simple localization strategy : Fixing the boundary condition equal to $n$ and restricting to realizations of $\phi$ which does not display long level lines. This restriction has the effect of killing most of the long range correlations (in a sense it corresponds to considering a massive version of the SOS field). After this restriction is performed we split $\gL_N$ into fixed sized cells and estimate the contribution of each cell to the free energy by a second moment argument. While quite technical and requiring some fine understanding of the SOS model behavior to be implemented, the method is in spirit analogous to the one used in \cite{cf:GL3} for the lattice GFF.

\medskip

 In Section \ref{sec:daupper}, we prove a quantitative version of the upper bound on the free energy displayed in Theorem \ref{mainres}. This is in our opinion the most difficult and novel part of the proof. In order to show that there is no better strategy than the one used for the lower bound, we need an argument to prove that the field is locally flat in most regions of the space. This part of the proof requires a novel coarse graining argument, based on a ``large contour'' decomposition of the field.

 \medskip
 
 \subsubsection*{A remark about notation} In the proofs, in order to avoid the excessive use of indices, we use the letter $C$ for a generic constant which does not depend on the parameters $\alpha, h, \beta$ and $\go$ and whose value may change from an equation to another. 

\section{Technical preliminaries}\label{prelim}

\subsection{Contour representation}\label{contour}

We recall how to describe a function $\phi\in \gO_{\gL}$ using only its level lines.
The formalism of this section is identical to the one used in \cite{cf:part1, cf:HL}, and inspired by other papers making use of  contours to study the properties of the  SOS model
\cite{cf:ADM, cf:PLMST, cf:CM}.

\medskip

We let $(\bbZ^2)^*$ denote the dual lattice of $\bbZ^2$ (dual edges cross that of $\bbZ^2$ orthogonally in their midpoints).
Two adjacent edges  $(\bbZ^2)^*$ meeting at $x^*$ of are said to be \textit{linked} if they both 
lie on the same side of the line making an angle $\pi/4$ with the horizontal 
and passing through $x$. 
(see Figure \ref{linked}).
We define a \textit{contour sequence} to be  a finite sequence $(e_1,\dots,e_n)$ of distinct edges of $(\bbZ^2)^*$ which satisfies:
\begin{itemize}
 \item [(i)] For any $i=\lint 1,n-1\rint$, $e_i$ and $e_{i+1}$ have a common end point in $(\bbZ^2)^*$, $e_1$ and $e_{n}$ also have a common end point.
 \item [(ii)] If for $i\ne j$, if $e_i$, $e_{i+1}$, $e_j$ and $e_{j+1}$ meet at a common end point then  $e_i$, $e_{i+1}$ are linked and so are $e_j$ and $e_{j+1}$ (with the convention that $n+1=1$).
\end{itemize}
A \textit{geometric contour} $\tilde \gamma:=\{e_1,\dots,e_{|\tilde \gamma|}\}$ is a set of edges that forms a contour sequence when displayed in the right order.
The cardinality $|\tilde \gamma|$ of $\tilde\gamma$ is called the length of the contour.

\begin{rem}\label{sequence}
 Note that equivalently a contour sequence can be described by a sequence of lattice points $(x_1,\dots,x_n)$ such that $x_i\sim x_{i+1}$ for all $i\in \lint 1,n \rint$ (with the convention that $n+1=1$), which also has to satisfy some additional condition. This alternative description is useful when introducing a coarse grained version of contours in Section \ref{seccg}.
 \end{rem}

A \textit{signed contour} or simply \textit{contour} $\gamma=(\tilde \gamma,\gep)$ 
is a pair composed of a geometric contour and a sign $\gep\in \{+1,-1\}$. We let $\gep(\gamma)$ denote the sign associated with a contour $\gamma$, 
while with a small abuse of notation, 
$\tilde \gamma$ is used for the geometric contour associated with $\gamma$ when needed.
For $x^*\in (\bbZ^2)^*$ we write $x^*\in \gamma$ or $x^*\in \tilde \gamma$ when the point $x^*$ is visited by one edge of the geometric contour.

 \begin{figure}[ht]
\begin{center}
\leavevmode
\epsfysize = 3 cm
\epsfbox{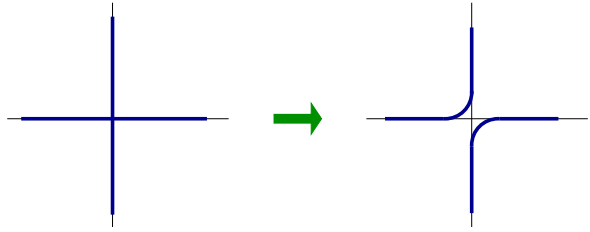}
\end{center}
\caption{\label{linked} 
The rule for splitting four edges meeting at one point into two pairs of linked edges.
To obtain the set of contours that separates $\{x \ : \ \phi(x)\ge h\}$ from   $\{x \ : \ \phi(x)< h\}$ for $h\in \bbZ$,
we draw all dual edges separating two sites $x$, $y$ such that $\phi(x)\ge h>\phi(y)$ and apply the above graphic rule for every dual vertex where four edges meet.
When several sets of level lines include the same contour, it corresponds to a cylinder of intensity $2$ or more for $\phi$.}
\end{figure}

We let $\bar \gamma$ denote the set of vertices of $\bbZ^2$ enclosed by $\tilde \gamma$.
We refer to $\bar \gamma$ as the \textsl{interior of} $\gamma$ and say that $|\bar \gamma|$ is the volume enclosed in the contour $\gamma$.
We let $\Delta_{\gamma}$, the neighborhood of $\gamma$, be the set of vertices of $\bbZ^2$ located either at a (Euclidean) distance $1/2$ from 
$\tilde\gamma$ 
(when considered as a subset of $\bbR^2$) 
or at a distance $1/\sqrt{2}$ from the meeting point of two non-linked edges.
We split the $\Delta_{\gamma}$ into two disjoint sets, the internal and the external neighborhoods of $\gamma$ 
(see Figure \ref{compa})
\begin{equation}\label{internextern}
\Delta^-_{\gamma}:=\Delta_{\gamma}\cap \bar \gamma \quad \text{ and } \quad  \Delta^+_{\gamma}:=\Delta_{\gamma}\cap \bar \gamma^{\cc}.
\end{equation}

\medskip

Given a finite set $\gL\subset \bbZ^2$ 
a contour $\gamma$ is said to be in $\gL$  if $\bar \gamma\subset \gL$.
We let $\cC$ denote the set of contours in $\bbZ^2$ and $\cC_{\gL}$ that of contours in $\gL$.
Given $\phi\in \gO_{\gL}$, we say that $\gamma \in \cC_{\gL}$  is a contour
for $\phi$ with boundary condition $n$, if there exists $k\ge 1$ such that 
\begin{equation}\label{cont+}
\min_{x\in \Delta^-_{\gamma}} \phi(x)= \max_{x\in \Delta^+_{\gamma}} \phi(x)+k\gep(\gamma).
\end{equation}
where in the above equation by convention we consider  that
$$\phi(x)=n \quad \text{ if } \quad x\in \gL^{\cc}.$$
The quantity $k$ appearing in \eqref{cont+} is called the \textit{intensity} of the contour and the
triplet $(\gamma,k)=(\tilde \gamma, \gep(\gamma), k)$
with $\gamma \in \cC$ and $k\in \bbN$ an intensity, is called a \textit{cylinder}. We say that  $(\gamma,k)$
is a cylinder for $\phi$ (with boundary condition $n$) if $\gamma$ is a contour of intensity $k$.
The cylinder function associated with $(\gamma,k)$ is defined on $\bbZ^2$ by
\begin{equation}\label{cylfunc}
\varphi_{(\gamma,k)}(x)=\gep(\gamma) k \ind_{\bar \gamma}(x).
\end{equation}
We use the notation $\hat \gamma$ to denote a generic cylinder and write $k(\hat \gamma)$ to denote its intensity. With some small abuse of notation, $\tilde \gamma$, $\gamma$ and $\bar \gamma$ denote the geometric contour, contour, and contour interior associated with $\hat \gamma$.
We let $\hat \Upsilon_n(\phi)$ denote the set of cylinders for $\phi$ with boundary condition $n$ and $\Upsilon_n(\phi)$ the corresponding set of contours.\\

\medskip

\noindent We say that $\gL$ is a \textit{simply connected} subset of $\bbZ^2$, if it can be expressed as the interior of a contour,
that is, if
\begin{equation} \label{connectedness}
 \exists \gamma_{\gL}\in \cC, \quad \bar\gamma_{\gL}=\gL.
\end{equation}
Note that, when $\gL$ is simply connected, 
an element
$\phi\in \gO_{\gL}$ is uniquely characterized by its cylinders. More precisely, we have  
\begin{equation}\label{cylinder}
\forall x\in \gL, \quad \phi(x):= n+\sum_{\hat \gamma \in \hat \Upsilon_n(\phi)}  \varphi_{\hat \gamma}(x).
\end{equation}
Furthermore, the reader can check that
\begin{equation}\label{express}
 \cH^n_{\gL}(\phi)=\sum_{\hat \gamma \in \hat \Upsilon_n(\phi)} k(\hat \gamma)|\tilde \gamma|.
\end{equation}
Of course not every set of cylinder is of the form $\hat \Upsilon_n(\phi)$ and we must introduce a notion of compatibility
which characterizes the ``right" sets of cylinder.

\medskip

Two cylinders $\hat \gamma$ and  $\hat \gamma'$ are said to be \textit{compatible} if they are
cylinders for the function $\varphi_{\hat \gamma}+\varphi_{\hat \gamma'}$.
This is equivalent to the three following conditions  being satisfied :
 (see Figure \ref{compa})
\begin{itemize}
 \item [(i)] $\tilde \gamma\ne \tilde\gamma'$ and $\bar \gamma \cap  \bar \gamma' \in\{\emptyset,\bar \gamma, \bar \gamma'\}$.
  \item [(ii)]  If $\gep= \gep'$ and $\bar\gamma\cap \bar \gamma'= \emptyset$, then
 then $\bar \gamma'\cap \Delta^+_{\gamma}=\emptyset$ .
 \item [(iii)] If $\gep\ne \gep'$ and $\bar \gamma'\subset \bar \gamma$ (resp.\ $\bar \gamma\subset \bar \gamma'$) then 
 $\bar \gamma'\cap \Delta^-_{\gamma}=\emptyset$  (resp. $\bar \gamma\cap \Delta^-_{\gamma'}=\emptyset$).
\end{itemize}
This first condition simply states that compatible contours do not cross each-other.
The conditions $\bar \gamma'\cap \Delta^+_{\gamma}=\emptyset$ and  $\bar \gamma'\cap \Delta^-_{\gamma}=\emptyset$  in  $(ii)$ and $(iii)$  
can be reformulated as: $\tilde \gamma$ and $\tilde \gamma'$ do not share edges, and if both $\tilde \gamma$ and $\tilde \gamma'$ possess
two edges adjacent to one vertex 
$x^*\in (\bbZ^2)^*$ 
 then the two edges in $\gamma$ are linked and so are those in $\gamma'$.
 
 \medskip
 
  \begin{figure}[ht]
\begin{center}
\leavevmode
\epsfxsize =5 cm
\epsfbox{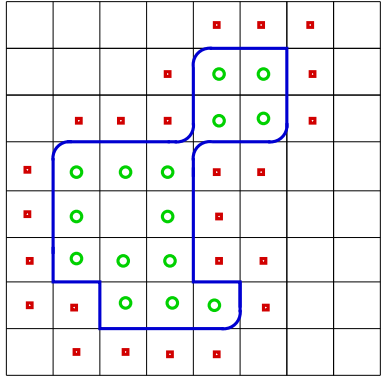}
\end{center}
\caption{\label{compa} 
A contour $\gamma$ represented with its internal ($\Delta^-_{\gamma}$ circles) and external ( $\Delta^+_{\gamma}$, squares) neighborhood.
To be compatible with $\gamma$, a contour $\gamma'$ of the same sign such that $\bar \gamma'\cap \bar \gamma= \emptyset$ cannot enclose any squares.
A compatible contour of opposite sign enclosed in $\gamma$ (such that $\bar \gamma'\subset \bar \gamma$) cannot enclose any circles.
}
\end{figure}

Note that the compatibility of two cylinders does not depend on their respective intensity, so that the notion can naturally be extended to signed contours :
the contours  $\gamma$ and $\gamma'$ are said to be compatible (we write $\gamma \mid \gamma'$) if the cylinders $(\gamma,1)$ and $(\gamma',1)$ are.
Two distinct non-compatible contours are said to be \textit{connected} (we write $\gamma \perp \gamma'$).

 \begin{figure}[ht]
\begin{center}
\leavevmode
\epsfysize = 7 cm
\epsfbox{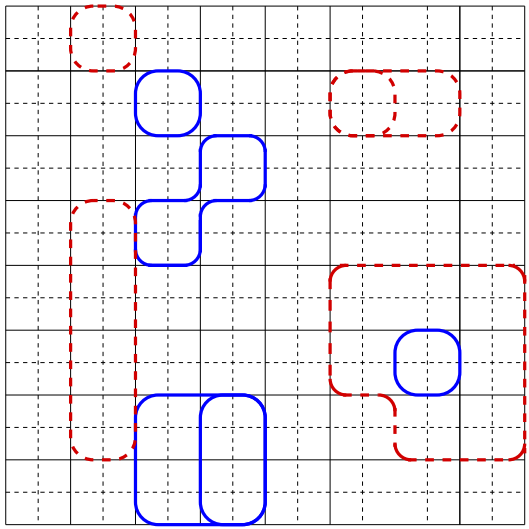}
\end{center}
\caption{\label{grid} 
A compatible collection of contour on the dual lattice (the primal lattice is represented in dotted lines).
Contours of different signs are displayed in different colors (red-dotted/blue-solid).
}
\end{figure}

\medskip

A (finite or countable) collection of cylinders (or of signed contours) 
is said to be a compatible collection if its elements are pairwise compatible (see Figure \ref{grid}).
The reader can check by inspection the following claim which establishes that the set of compatible collections of cylinders is in bijection with $\gO_{\gL}$  (simple connectivity is required to avoid having level lines enclosing holes).

\begin{lemma}\label{toott}
If $\gL$ is simply connected, then for any $\phi\in \gO_{\gL}$, $\hat \Upsilon_n(\phi)$ is a compatible collection of cylinders
and reciprocally, if $\hat \gG \subset \hat \cC_{\gL}$ is a compatible collection of cylinder in $\gL$
then its elements are the cylinders of the function 
$\sum_{\hat \gamma\in \hat \gG} \varphi_{\hat \gamma}.$
\end{lemma}

Using \eqref{express} and the contour representation above, we can rewrite the partition function $\cZ_{\gL,\gb}$ in a new form.
We let $\cK(\gL)$ and $\hat \cK(\gL)$ denote the set of compatible collections of contour and cylinders in $\gL$.
We have 
\begin{equation}
 \cZ_{\gL,\gb}= \sum_{\hat \gG \in \hat \cK(\gL)} \prod_{\hat \gamma \in \hat \gG} e^{-k(\hat \gamma)\gb|\tilde \gamma|}.
\end{equation}
Summing over all the possible intensities, we obtain
\begin{equation}\label{contourdecomp}
 \cZ_{\gL,\gb}=\sum_{\gG \in  \cK(\gL)} \prod_{\gamma \in  \gG} \frac{1}{e^{\gb|\tilde \gamma|}-1}.
\end{equation}
This rewriting of the partition function entails that under $\bP_{\gL,\beta}$ the distribution of $\Upsilon_n(\phi)$ is a product measure conditioned to compatibility. In particular as the event ``being compatible'' is a decreasing event for {\blue the inclusion relation in the power set $\cP(\cC_{\gL})$}, Harris inequality \cite{cf:Harr}  implies that $\bP_{\gL,\beta}$ is stochastically dominated by this product measure.
To make this more explicit, let  $\bQ_{\gL,\gb}$ be the distribution of a random element $\chi\in \cP(\cC_{\gL})$ (the power set of $\cC_{\gL}$) under which
the variables $\chi(\gamma):=\ind_{\{\gamma\in \chi\}}$ are independent and 
\begin{equation}\label{defgroq}
\bQ_{\gL,\gb}\left( \gamma\in \chi \right)=e^{-\gb|\tilde \gamma|}. 
\end{equation}
We refer to \cite[Lemma 4.3 and 4.4]{cf:part1} for more details of the proof of the following statement.
\begin{lemma}\label{restrict}
The distribution of the cylinders in $\hat  \Upsilon_n(\phi)$ can be described as follows:
\begin{itemize}
 \item [(A)] We have 
\begin{equation}\label{condit}
 \bP^n_{\gL,\gb}[\Upsilon_n(\phi) \in \cdot]:= \bQ_{\gL,\gb}\left[ \chi \in \cdot \ | \  \chi \text{ is a compatible collection } \right].
\end{equation}
In particular the distribution of $\Upsilon_n(\phi)$ is stochastically dominated (for the inclusion) by $\bQ_{\gL,\gb}$.
\item[(B)]
Conditionned to the realization $\Upsilon_n(\phi)$ the intensities $(k(\hat \gamma))_{\hat \gamma \in \hat \Upsilon_n(\phi)}$ are independent geometric variable of parameter $e^{-\beta |\tilde \gamma|}$ 
\end{itemize}

\end{lemma}
\noindent We end up this section by introducing some notation an terminology which we use in the remainder of the paper.

\medskip
\noindent\emph{SOS measure with contour restrictions}:
For a fixed  set of contours $\bL\subset \cC_{\gL}$, let us define the measure 
$\bP^n_{\bL,\gL,\gb}$ which is obtained by conditioning all the contours of $\phi$ to belong to $\bL$ 
\begin{equation}
 \bP^n_{\bL,\gL,\gb}:= \bP^n_{\gL,\gb}\left[ \cdot \ | \  \Upsilon_n(\phi)\subset \bL \right]
\end{equation}
{\blue The proof of  Lemma \ref{restrict}} extends to   $\bP^n_{\bL,\gL,\gb}$ in the sense  that  under the measure $\bP^n_{\bL,\gL,\gb}$,  the distribution of $\Upsilon_n(\phi)$ is stochastically dominated by that of $\chi \cap \bL$ under $\bQ_{\gL,\beta}$ 
(and thus also by that of $\chi$). {\blue This is because the distribution $\bP^n_{\gL,\gb}\left[ \Upsilon_n(\phi) \in \cdot \ | \  \Upsilon_n(\phi)\subset \bL \right]$ is a product distribution on $\cP(\bL)$  constrained to a decreasing event for the inclusion in $\cP(\bL)$ (compatibility).}

\medskip

\noindent\emph{External contours}: Given $\gG$ a compatible collection of contours and $\gamma \in \gG$, we say that $\gamma$ is an \textit{external contour} in $\gG$
if $\bar \gamma$ is maximal in $\gG$ for the inclusion ({\blue this time we are considering the inclusion relation in $\bbZ^2$}), that 
is 
\begin{equation}\label{defexternal}
 \forall \gamma'\in \gG, \quad \bar \gamma' \subset \bar\gamma \text{ or } \bar \gamma' \cap \bar \gamma= \emptyset.
\end{equation}
We let $\Upsilon^{\mathrm{ext}}_n(\phi)$ denote the set of external contours in $\Upsilon_n(\phi)$.
{\blue For $A\subset \bbZ^d$ we let $\phi \restrict_{A}$ denote the restriction of $\phi$ to the set $A$.}
Note that under $\bP^n_{\gL,\gb}$ 
(and also under  $\bP^n_{\bL,\gL,\gb}$ for arbitrary $\bL$),
conditionned to the realization of $\Upsilon^{\mathrm{ext}}_n(\phi)$ the restrictions of $\phi$ to the interior of external contours $\phi \restrict_{\bar \gamma}$,  $\gamma \in \Upsilon^{\mathrm{ext}}_n(\phi)$  are independent.

\subsection{Peak probabilities estimates for the SOS measure}\label{sossec}

We introduce here a result   concerning the asymptotic probability of observing high points for $\phi$ under  $\bP_{\bL,\gL,\gb}$.
Having estimates which are also valid with a contour restriction is of crucial importance in our proofs.
To state the result, we need to introduce the following notion of distance between a vertex and the complement of a finite set of contours
\begin{equation}\label{deffdiss}
d(x,\bL^{\cc}) := \min\{ \gamma \in \cC\setminus \bL, \max_{y\in \bar \gamma} |x-y|\}.
\end{equation}
Note that if $ \bL\supset \cC_{\gL}$ then $d(x,\bL^{\cc})$ is larger than the distance from $x$ to $\gL^{\cc}$.
The following estimates are proved in Appendix \ref{lazt}.
\begin{proposition}\label{rouxrou}
There exists a constant $C$ such that if $\gb\ge \beta_0$  sufficiently large, then such for any choice 
$\gL$, $\bL$
\begin{itemize}
 \item [(i)] For any $x\in \gL$ for any $n\ge 0$ we have
 \begin{equation}\label{zups}
 \bP_{\bL,\gL,\gb}[\phi(x)\ge n]\le C e^{-4\gb n}.
 \end{equation}
 If in addition,  the positive contour of length $4$ surrounding $x$ (which we denote by $\gamma_x$ is in $\bL$ we have
 \begin{equation}\label{czups}
   \bP_{\bL,\gL,\gb}[\phi(x)= n]\ge C^{-1} e^{-4\gb n}.
 \end{equation}

  \item [(ii)] For any $x,y\in \gL$ for any $n\ge 0$ we have
 \begin{equation}\label{zup2s}
 \bP_{\bL,\gL,\gb}[\min(\phi(x), \phi(y))\ge n]\le C e^{-6\gb n}.
 \end{equation}
  \item [(iii)] For any $x\in \gL$ for any $n\ge 0$ we have
 \begin{equation}\label{zup3s}
\left| \bP_{\bL,\gL,\gb}[\phi(x)\ge n]-\theta_1 e^{-4\beta n}\right| \le C\left( e^{-6 \beta n}+ e^{- d(x,\bL^{\cc})} \right). 
 \end{equation}
\item[(iv)] For any $x_1,\dots, x_k$, we have for $n\ge 0$
\begin{equation}\label{lotzof}
 \bP_{\bL,\gL,\gb}\left[\min_{i\in \lint 1, k\rint} \phi(x_i)\ge n\right]\le C_k e^{- 3\gb \sqrt{k}}.
 \end{equation}
\end{itemize}

\begin{rem}
Let us discuss here about the requirement on $\beta$ for our estimates to be valid.
 The proof of \eqref{zups}-\eqref{zup2s}  and \eqref{lotzof} are essentially based on a Peierls-type argument \cite{cf:Pei} and does not require more than the following condition
 \begin{equation*}
   \sum_{\{ \gamma\in \cC  \ : \ {\bf 0\in \bar \gamma}\}} e^{-\beta |\tilde \gamma|} <\infty
 \end{equation*}
which is valid when $\beta\ge 1$. 
On the contrary \eqref{zup3s} requires the convergence of the cluster expansion associated with the contour decomposition which is a more stringent condition. Looking at the discussion in  \cite[pp 493]{cf:KP}
and considering that the number of contours of length $n$ containing ${\bf 0}$ is bounded above by $n 3^{n-2} \le 4^n$ and that the weight associated with contours satisfy for $\beta\ge 1$,
$$(e^{\beta |\tilde \gamma|}-1)^{-1}\le  e^{-\frac{9\beta}{10} |\tilde \gamma|},$$
 we obtain that our estimates are valid as long as 
 $$\beta \ge \frac{10}{9}(\log 4 +1.6).$$
Hence in all our statements, $\beta$ sufficiently large can be replaced by $\beta \ge 3.5$. 
 \end{rem}

\end{proposition}

\section{Lower bound on the free energy}\label{sec:dalower}

\subsection{Result and decomposition of the proof}

The object of this section is to prove a quantitative version of the lower bound part of \eqref{crooti}.
\begin{proposition}\label{dalower}

There exists  $\gep$ such that for all $\gb$ sufficiently large and all $\alpha>0$, there exists $C(\alpha,\beta)>0$ such that for all $h>0$ 
  \begin{equation}\label{zook}
\bar \tf_{\beta}(\alpha,h) \ge  G_{\beta}(\alpha,h)- C(\alpha,\beta) h^{2+\gep}.
 \end{equation}
\end{proposition}
\noindent Note that as we have $G_{\beta}(\alpha,h)\stackrel{h\to 0+}{\asymp} h^2$, \eqref{zook} implies immediately that   
$$\liminf_{h\to 0+} \frac{\bar \tf_{\beta}(\alpha,h)}{G_{\beta}(\alpha,h)}\ge 1.$$
\noindent Let us also observe that it is enough to prove the statement for $h\le h_0(\alpha,\beta)$ where $h_0$ can be chosen arbitrarily small. Indeed by tuning the value of $C(\alpha,\beta)$ we can make the r.h.s.\ of \eqref{zook} negative for all $h\ge h_0$. We can make the constant $\gep$ in Proposition \ref{dalower} explicit (the reader can check reading through the proof that $\gep=1/100$ works), but we choose to work with a letter for the sake of readability.

\medskip

\subsubsection*{Decomposition of the proof of Proposition \ref{dalower}}
In this whole section we fix $n=n_G=n_{G}(h)$ to be the maximizer in \eqref{optimize} ({\blue
there is a geometric sequence of values for $h$ for which \eqref{optimize}  admits two maximizers, in the case we select the smallest one}). We have 
\begin{equation}\label{lexpression}
n_G(h)= \max\left(0, \left\lceil \frac{\log \left( \theta_1 (e^{\gl(2\alpha)-2\gl(\alpha)}-1)/(2h)\right)}{ 4\beta} \right\rceil \right)
\end{equation}
In particular $n_G(h)$ is asymptotically equivalent to $|\log h|/4\beta$ 
and as we are interested in the small values of  $h$, we can think of $n$ as large.
Recalling the definition of the reduced partition function \eqref{reduced}, we are going to show that for $N\ge N_0(h)$ sufficiently large  we have
  \begin{equation}\label{zook2}
  \frac{1}{N^2}\bbE \log Z^{n,h,\alpha,\go}_{\gb,N} \ge G_{\beta}(\alpha,\beta)- C(\alpha,\gb) h^{2+\gep}.
 \end{equation}
Taking the limit of \eqref{zook2} when $N$ tends to infinity we obtain \eqref{zook}.
We split the proof of the result in two main steps and three lemmas.
Our first step get rid of the possibility of having large contours. We show that $ Z^{n,h,\alpha,\go}_{\gb,N}$ can be replaced by
 $$\tilde Z^{n,h,\alpha,\go}_{\gb,N}:= \tilde \bE^n_{N,\gb}\left[ e^{\sum_{x\in \gL_N} (\alpha\go_x-\gl(\alpha)+h )\delta_x} \right],$$ 
where $$\tilde \bP^n_{N,\gb}:=\bP^n_{N,\gb}[ \cdot  \ | \ \cA_{N,n}],$$
is the SOS measure  with boundary condition 
$n$, conditioned  not to display contours of length larger than $n^4$.
$$\cA_{N,n}:= \left\{  \forall \gamma \in \Upsilon_n(\phi),  |\tilde \gamma|\le n^4  \right\}.$$
The choice of $n^4$ for the threshold defining large contours is quite arbitrary (the only requirements are subexponential growth in $n$ and being of a larger order of magnitude than $n^2$) and  but turns out to be convenient in the computation.
The following lemma implies that 
$ \tilde Z^{n,h,\alpha,\go}_{\gb,N}$ is a good approximation for $Z^{n,h,\alpha,\go}_{\gb,N}$.
\begin{lemma}\label{zerro}
We have for all $\gb\ge 2\log 3$, and all integers $N$ and $n$ 
\begin{equation}
 \bP^n_{N,\gb}\left[\cA_{N,n} \right]\ge \exp\left(- 4 e^{-\gb n^4/2} N^2 \right).
\end{equation}
In particular we have 
\begin{equation}\label{chankliche}
 \log Z^{n,h,\alpha,\go}_{\gb,N}\ge 
 \log \tilde Z^{n,h,\alpha,\go}_{\gb,N} -4 N^2 e^{-\gb n^4/2}.
\end{equation}
\end{lemma}
\noindent Note that with our choice for $n=n_G(h)$, we have $4 e^{-\gb n_G^4/2}\le h^3$ for $h$ sufficiently small and thus it is sufficient to prove  \eqref{zook2} with $ Z^{n,h,\alpha,\go}_{\gb,N}$ replaced by
$ \tilde Z^{n,h,\alpha,\go}_{\gb,N}$.
The second step is more delicate and relies on a coarse graining argument. We wish to divide our system into cells in order to factorize the partition function.
In order to obtain a factorization we must condition on the realization of the set of contours which intersect several cells.

\medskip

We set $M= h^{-\frac{1}{50}}$, $N=kM$ where $k\in \bbN$ is arbitrary. Imposing that $N$ is a multiple of $N$ is by no mean restrictive. Indeed, since the limit of the l.h.s.\ exists, it is in fact sufficient to prove \eqref{zook2} along a subsequence.
We  divide $\gL_N= \lint 1,N\rint$ in $k^2$ boxes.
We let 
$$\bbH=\bbH_M:=\{ (x_1,x_2)\in (\bbZ^2)^* \ : \  M |(x_1-1/2) \text{ or } M |(x_2-1/2) \} $$
denote the grid in the dual lattice which splits $\gL_N$ into cells of side-length $M$. We index these cells by  $z\in \lint 0,k-1\rint^2$,
and set $\cC_z:= \gL_M +Mz$. 
Recalling the definition of an external contour \eqref{defexternal}, we are going to consider the set of countours intersecting $\bbH$ (here we consider the vertex intersection)
$$U(\phi):=\{ \gamma\in \Upsilon_n^{\ext}(\phi) \ : \ \gamma \cap \bbH \ne \emptyset \}.$$
Note that the contours of $U(\phi)$ naturally provides a partition of the box $\gL_N$. If $U(\phi)=\gG$, setting 
$\bar \gG:= \bigcup_{\gamma \in \gG} \bar \gamma$,
we let, for $z\in \lint 0,k-1\rint^2$,  $\cB_z$ denote the set of sites in $\cC_z$ which are not enclosed in a contour of $\gG$,
$$\cB_z:= \cC_z\setminus \bar \gG.$$
Note that $\cB_z$ is not necessarily connected.
We can observe that 
\begin{equation}\label{patricion}
\gL_N=\left( \bigcup_{z\in \lint 1,k\rint} \cB_z\right) \cup \left( \bigcup_{\gamma \in \gG } \bar \gamma \right),
\end{equation}
and that the union is disjoint. 
Moreover using the contour decomposition, one can check  
that conditioned on $U(\phi)=\gG$, the restrictions of $\phi$  to the regions of the partition in the l.h.s.\ of \eqref{patricion} form an independent family.
Hence we have 
\begin{multline}
\tilde \bE^n_{N,\gb}\left[e^{\sum_{x\in \gL_N} (\alpha \go_x- \gl(\alpha)+h)\delta_x}  |  U(\phi)=\gG \right]\\
 = \prod_{z\in \lint 0,k-1\rint^2 } \tilde \bE^n_{N,\gb}\left[e^{\sum_{x\in \cB_z} (\alpha \go_x- \gl(\alpha)+h)\delta_x}  \ | \ U(\phi)=\gG \right] \\ 
 \times \prod_{\gamma \in \gG } \tilde \bE^n_{N,\gb}\left[e^{\sum_{x\in \bar \gamma} (\alpha \go_x- \gl(\alpha)+h)\delta_x} \ | \ U(\phi)=\gG \right],
\end{multline}
and as a consequence
\begin{multline}\label{troux}
  \bbE \log \tilde Z^{n,h,\alpha,\go}_{\gb,N}
  \ge \min_{\gG} \Bigg[ \sum_{z\in \lint 0,k-1\rint^2 } \bbE  \log \tilde \bE^n_{N,\gb}\left[e^{\sum_{x\in \cB_z} (\alpha \go_x- \gl(\alpha)+h)\delta_z}  \ | \ U(\phi)=\gG \right]\\
  + \sum_{\gamma \in \gG } \bbE \log\tilde \bE^n_{N,\gb}\left[e^{\sum_{x\in \bar \gamma} (\alpha \go_x- \gl(\alpha)+h)\delta_z} \ | \ U(\phi)=\gG\right] \Bigg],
\end{multline}
where the minimum is taken over all possible realization of $U(\phi)$.
To conclude we  obtain a lower bound on each term of the two sums in the l.h.s.\ above, yielding an estimate which is uniform in the realization of $U(\phi)$.
As the cells $\cB_z$ cover much more area than the interior of contours in $U(\phi)$, we need to show that sites inside $\cB_z$ gives a  contribution per site very close to $G_{\beta}(\alpha,h)$, while for the restriction to $\bar \gamma$, showing that the contribution of each site is of order $h^2$ in absolute value is sufficient. This the content of the two following lemmas.

\begin{lemma}\label{firstlemma}
 There exists a constant $C(\alpha,\beta)$ such that for any $h\in (0, 1]$, for any $\gamma \in \cC_{\gL_N}$ with  $|\tilde \gamma|\le  (n_G(h))^4$ we have

 \begin{equation}
\bbE \log \tilde \bE^n_{N,\gb}\left[e^{\sum_{x\in \bar \gamma} (\alpha \go_x- \gl(\alpha)+h)\delta_z} \ | \ \gamma \in \Upsilon_n^{\ext}(\phi) \right]\ge -C|\bar \gamma| h^2.
 \end{equation}

\end{lemma}

\begin{lemma}\label{secondlemma}
There exists a constant $h_0(\alpha,\beta)>0$ such that if $h\in (0, h_0(\alpha,\beta)]$ then
we have, for any possible realization of $U(\phi)$ and any $z\in \lint 0,k-1\rint$

 \begin{equation}
 \bbE \log \tilde \bE^n_{N,\gb}\left[e^{\sum_{x\in \cB_z} (\alpha \go_x- \gl(\alpha)+h)\delta_x}  \ | \ U(\phi)=\gG \right]\ge   M^2\left[  G_{\gb}(\alpha,h)-h^{2+\gep}\right].
 \end{equation}

\end{lemma}
\noindent Before proving the above results, let us show how they permit to conclude our proof of Proposition \ref{dalower}.

\begin{proof}[Proof of Proposition \ref{dalower}]
In view of \eqref{zook2},\eqref{chankliche} and  \eqref{troux}, we only need to show that for every choice of $\gG$ we have 
\begin{multline}\label{tooproove}
  \sum_{z\in \lint 0,k-1\rint^2 } \bbE  \log \tilde \bE^n_{N,\gb}\left[e^{\sum_{x\in \cB_z} (\alpha \go_x- \gl(\alpha)+h)\delta_z}  \ | \ U(\phi)=\gG \right]\\
  + \sum_{\gamma \in \gG } \bbE \log\tilde \bE^n_{N,\gb}\left[e^{\sum_{x\in \bar \gamma} (\alpha \go_x- \gl(\alpha)+h)\delta_z} \ | \ U(\phi)=\gG\right]  \\ \ge N^2  G_{\gb}(\alpha,h)- Ch^{2+\gep}.
\end{multline}
It is a direct consequence of the contours construction that
\begin{equation}
\tilde \bE^n_{N,\gb}\left[e^{\sum\limits_{x\in \bar \gamma} (\alpha \go_x- \gl(\alpha)+h)\delta_x} \ | \ U(\phi)=\gG \right]=
\tilde \bE^n_{N,\gb}\left[e^{\sum\limits_{x\in \bar \gamma} (\alpha \go_x- \gl(\alpha)+h)\delta_x} \ | \ \gamma \in \Upsilon_n^{\ext}(\phi) \right].
\end{equation}
Hence taking the expectation with respect to $\go$ in \eqref{troux} and using the lemmas to evaluate each term of the sum we obtain that the r.h.s.\ of \eqref{tooproove} is larger than

\begin{equation}
 N^2\left[  G_{\gb}(\alpha,h)-h^{2+\gep}\right]-C h^2 \tilde \bE^n_{N,\gb}\sum_{\gamma \in \gG } |\bar \gamma| .
\end{equation}
Now the sites enclosed by some $\gamma\in U(\phi)$ are all  located at distance $n^4$ from $\bbH$.
Hence we have (provided that $\gep< 1/50$), for all $h$ sufficiently small 
 $$ \sum_{\gamma \in U(\phi) } |\bar \gamma| \le 4 n^4 k^2 M   = \frac{4 n^4 N^2}{M}\le N^2 h^{\gep}.$$
 This is sufficient to conclude.
 
 \end{proof}
\subsection{Proof of Lemma \ref{zerro}}
The event $\cA_{N,n}$ is decreasing for the inclusion.
Thus applying Lemma \ref{restrict} we have 
\begin{equation}
  \bP^n_{N,\gb}\left[\cA_{N,n} \right]
  \ge \bQ_{N,\beta}[ \ \forall \gamma \in \chi, \ |\tilde \gamma|  \le n^4 ].
\end{equation}
Now, if $\cC_N$ denotes the set of contours in the box $\gL_N$ the latter probability is exactly equal to 
\begin{equation}
 \prod_{\gamma \in \cC_{N}}(1-e^{-\beta|\tilde \gamma|})\ge \exp\left( -2 \sum_{\gamma \in \cC_N} e^{-\beta |\tilde \gamma|} \right)
\end{equation}
where we used $1-x\ge e^{-2x}$ which is valid for $x\in (0,1/2)$.
We conclude by observing that, as the number contour of length $m$ in $\lint 1,N\rint^2$ is bounded above by $N^2 3^m$ 
\begin{equation}
 \sum_{\gamma \in \cC_N} e^{-\beta |\tilde \gamma|} \le  N^2 \frac{e^{-(\beta-\log 3) n^4}}{1- e^{-(\beta-\log 3)}},
\end{equation}
which allows to conclude. 
 
 \qed
\subsection{Proof of Lemma \ref{firstlemma}}

For both Lemma \ref{firstlemma} and \ref{secondlemma} our strategy is to rely on second moment computation together with a Taylor expansion.
Let us first describe the distribution of $\phi$ restricted to
$\bar \gamma$ after the conditioning. We have (recalling \eqref{internextern}) from Lemma \ref{restrict}
\begin{equation}\label{lunoulaut}
 \tilde \bP^n_{N,\gb}\left[\phi\restrict _{\bar \gamma} \in \cdot | \ \gamma \in \Upsilon^{\ext}_n \right]= \begin{cases} \tilde \bP^{n-1}_{\bar \gamma,\gb}[ \cdot \ | \ \forall x\in \Delta_{\gamma}^-, \phi\le n-1] \quad \text{ if } \gep(\gamma)=-1, \\ 
 \tilde \bP^{n+1}_{\bar \gamma,\gb}[ \cdot \ | \ \forall x\in \Delta_{\gamma}^-, \phi\ge n+1] \quad \text{ if } \gep(\gamma)=+1.
\end{cases}
\end{equation}
where the tilde on the r.h.s.\ is present to remind ourselves that we are conditioning on having no contour of length more than $n^4$.
In order to have a better control on the second moment of the partition function, we restrict the computation to the set of surfaces $\phi$ which display only a small number of contacts.  We introduce the event 
\begin{equation}
 \cD^{\kappa}_{\gamma}:= \{ \phi \ : \ \sum_{x\in \bar \gamma} \delta_x \le \kappa\}. 
\end{equation}
{\blue We fix the value of $\kappa$ equal to $100$ (and drop the dependence in $\kappa$ in the notation) but keep the letter  $\kappa$ in the computation for better readability}.
As in both cases in the r.h.s. of  \eqref{lunoulaut},
the measure is of the type $\bP^{n\pm 1}_{\bL,\gL,\beta}$
(the conditioning corresponds to prohibiting positive (if $\gep(\gamma)=-1$) or negative (if $\gep(\gamma)=+1$) contours which enclose elements of $\gD^-_\gamma$).
Hence using Proposition \ref{rouxrou} and a union bound we have (recall that  $|\bar \gamma|\le n^{16}$ from the  restriction on the contours' length)
\begin{equation}\label{leptit}
  \tilde \bP^n_{N,\gb}\left[  \cD_{\gamma}^{\cc}\ | \ \gamma \in \Upsilon^{\ext} \right]\le \binom{|\bar \gamma|}{\kappa} e^{-30\gb (n-1)} \le h^3, 
\end{equation}
We define $\mu_{\gamma,h}$ the probability on $\gO_{N}$ defined by
\begin{equation}
 \mu_{\gamma,h}(A):= \frac{\tilde \bE^n_{N,\gb}\left[\ind_{\cD_{\gamma}\cap A}e^{h \sum_{x\in \bar \gamma} \delta_x} \ | \ \gamma \in \Upsilon^{\ext}  \right]}{\tilde \bE^n_{N,\gb}\left[ e^{h \sum_{x\in \bar \gamma} \delta_x}\ind_{\cD_{\gamma}} \ | \ \gamma \in \Upsilon^{\ext} \right]}.
\end{equation}
We have 
\begin{multline}
\log  \tilde \bE^n_{N,\gb}\left[e^{\sum_{x\in \bar \gamma} (\alpha \go_x- \gl(\alpha)+h)\delta_x}\ind_{\cD_{\gamma}} \ | \ \gamma \in \Upsilon^{\ext} \right]\\
=\log \tilde \bE^n_{N,\gb}\left[e^{\sum_{x\in \bar \gamma} h\delta_x}\ind_{\cD_{\gamma}} \ | \ \gamma \in \Upsilon^{\ext} \right]
 +\log \mu_{\gamma,h}\left(e^{\sum_{x\in \bar \gamma}\left(\alpha \go_x- \gl(\alpha)\right)\delta_x} \right).
\end{multline}
Using \eqref{leptit}, when $h$ is sufficiently small  the first term is larger than
$$\log  \tilde \bP^n_{N,\gb}\left[\cD_{\gamma} \ | \ \gamma \in \Upsilon^{\ext} \right]\ge -h^3.$$
Now we have to estimate the expectation of the second term.
Combining \eqref{lunoulaut} with Proposition \ref{rouxrou} we obtain that for all $x\in \bar \gamma$,
\begin{equation}
 \bE^n_{N,\gb}\left[\delta_x \ | \ \gamma \in \Upsilon^{\ext} \right]
 \le C e^{-4\gb(n-1)},
\end{equation}
for some positive constant $C$.
Using the expression for the density and the assumption that $|\bar \gamma|\le n^{16} \le C |\log h|^{16}$ we obtain for sufficiently small $h$, a similar estimate under    $\mu_{\gamma,h}$
\begin{equation}\label{unifbound}
 \mu_{\gamma,h}(\delta_x)\le \frac{C e^{-4\gb(n-1)+h |\bar\gamma|}}{\tilde \bP^n_{N,\gb}\left[  \cD_{\gamma}\ | \ \gamma \in \Upsilon^{\ext} \right]}\le C' e^{-4\gb(n-1)}.
\end{equation}
Then by Markov's inequality we have
\begin{equation}
\mu_{\gamma,h}\left(e^{\sum_{x\in \bar \gamma}(\alpha \go_x- \gl(\alpha))\delta_x} \right)\ge  \mu_{\gamma,h}\left( \sum_{x\in \bar \gamma} \delta_x=0 \right)\ge 1-C' |\bar \gamma| e^{-4\gb(n-1)} \ge \frac{1}{2}.
\end{equation}
We can use the inequality $\log x\ge (x-1)-(x-1)^2$ (valid for $x\ge 1/2$) and obtain
\begin{equation}
 \bbE\left[ \log  \mu_{\gamma,h}\left(e^{\sum_{x\in \bar \gamma}(\alpha \go_x- \gl(\alpha))\delta_x} \right)\right] \ge - \bbE \left[ \mu_{\gamma,h}\left(e^{\sum_{x\in \bar \gamma}\alpha \go_x- \gl(\alpha)}-1 \right)^2 \right].
\end{equation}
The average w.r.t.\ $\go$  can be computed explicitly, we obtain
\begin{equation}
 \bbE \left[ \mu_{\gamma,h}\left(e^{\sum_{x\in \bar \gamma}\alpha \go_x- \gl(\alpha)}-1\right)^2 \right]=
\mu_{\gamma,h}^{\otimes 2}\left[ e^{[\gl(2\alpha)-2\gl(\alpha)]\sum_{x\in \bar \gamma} \delta^{(1)}_x\delta^{(2)}_x} \right]-1.
\end{equation}
From the definition of $\mu_{ \gamma,h}$ we have 
$\sum_{x\in \bar \gamma} \delta^{(1)}_x\delta^{(2)}_x\le \kappa$,  
with probability $1$.
Hence using the fact that for any $u \in [0, \kappa(\gl(2\alpha)-2\gl(\alpha))]$ 
we have 
\begin{equation}
 e^{u}-1\le u e^{\kappa(\gl(2\alpha)-2\gl(\alpha))}.
\end{equation}
Using \eqref{unifbound} and the expression for $n=n_G(h)$ given in  \eqref{lexpression}, we can conclude the proof as follows
\begin{multline}
 \mu_{\gamma,h}^{\otimes 2}\left[ e^{[\gl(2\alpha)-2\gl(\alpha)]\sum_{x\in \bar \gamma} \delta^{(1)}_x\delta^{(2)}_x} \right]-1
 \le e^{\kappa(\gl(2\alpha)-2\gl(\alpha))}\mu_{\gamma,h}^{\otimes 2}\left(\sum_{x\in \bar \gamma} \delta^{(1)}_x\delta^{(2)}_x\right)\\
 \le 4 e^{\kappa(\gl(2\alpha)-2\gl(\alpha))} |\bar \gamma| e^{-8\gb (n-1)} \le C(\alpha, \kappa,\beta)|\bar \gamma| h^2.
\end{multline}

\qed 
\subsection{Proof of Lemma \ref{secondlemma}}

The proof follows the same steps as that of Lemma \ref{firstlemma}, except that we must aim for  sharper  estimates.
We define $\tilde \mu_{z,h}$  a probability of $\gO_N$ by

\begin{equation}
 \tilde \mu_{z,h}(A):= \frac{\tilde \bE^n_{N,\gb}\left[ \ind_{\{\phi  \  \restrict_{\cB_z}\in A \cap \cD_z \}}e^{h\sum_{x\in \cB_z} \delta_x}  \ | \ U(\phi)=\gG \right]}{\tilde \bE^n_{N,\gb}\left[ \ind_{\{\phi \ \restrict_{\cB_z}\in \cD_z\}}e^{h\sum_{x\in \cB_z} \delta_x}  \ | \ U(\phi)=\gG \right]}
\end{equation}
with ({\blue again $\kappa$ is set to be equal to $100$ but the letter is kept in the computation for better readability)}
\begin{equation}
 \cD_{z}:= \{ \phi \ : \ \sum_{x\in \cB_z} \delta_x \le \kappa\}.
\end{equation}
We are going to refine second moment argument of the second section. In particular we want to use the fact (cf.\ Proposition \ref{rouxrou}) that the probability of contact is 
close to $\theta_1 e^{-4\beta n}$ for most points inside $\cB_z$.
As these estimates are not valid close to the boundary of $\cB_z$, 
let us define $\cB'_z$ a subset of $\cB_z$ which includes only points in the bulk
$$ \cB'_z:= \lint n^4, M-n^4 \rint^2 +Mz  \subset \cB_z.$$
The point here is that from the definition, because contours in $\gG$ are all of diameter smaller than $n^4/2$ (length smaller than $n^4$) then all points in $\cB'_z$ are at a distance at least $n^4/2$ from the boundary of $\cB_z$.
The following estimates are easily deduced from Proposition \ref{rouxrou} (we include details at the end of the present subsection). 

\begin{lemma}\label{zetims}
 There exists  constants $C>0$ and $\beta_0>0$  such that for all $\beta>\beta_0$  we have
\begin{itemize}
 \item[(i)] For all $x\in \cB'_z$ we have
 \begin{equation}\label{toutt}
 |\tilde \mu_{z,h}(\delta_x)-\theta_1 e^{-4\gb n}|\le  C h^{3/2}.
\end{equation}
This estimate is also valid for $\tilde \mu_{z,0}$.
\item[(ii)] For all {\blue $x\in \cB_z$} we have 
\begin{equation}\label{touttt}
 \tilde \mu_{z,h}(\delta_x)\le C e^{-4\gb n}.
\end{equation}
This is also valid for $\tilde \mu_{z,0}$.
\item[(iii)] For all $x,y \in \cB_z$
\begin{equation}\label{mcdouble}
  \tilde \mu_{z,h}(\delta_x \delta_y)\le C e^{-6\gb n}.
\end{equation}

\end{itemize}
\end{lemma}
Let us now prove the result using the estimates above.
Note that we have 
 \begin{multline}\label{thebegin}
 \log \tilde \bE^n_{N,\gb}\left[ \ind_{\cD_z}e^{\sum_{x\in \cB_z} (\alpha \go_x- \gl(\alpha)+h)\delta_x}  \ | \ U(\phi)=\gG \right]\\
 =
  \log \tilde \bE^n_{N,\gb}\left[\ind_{\cD_z}e^{\sum_{x\in \cB_z}h\delta_x}  \ | \ U(\phi)=\gG \right]+ \log \tilde \mu_{z,h}\left(e^{\sum_{x\in \cB_z}(\alpha \go_x-\gl(\alpha))} \right).
 \end{multline}
The first term is equal to 
\begin{equation}\label{zeconti}
 \log   \tilde \bP^n_{N,\gb}\left[\cD_z \ | \ U(\phi)=\gG \right]
 +\log \tilde \mu_{z,0} (e^{\sum_{x\in \cB_z}h\delta_x}).
\end{equation}
Similarly to  \eqref{leptit} using that 
$M=h^{-\frac{1}{50}}$ and recalling that $\kappa=100$, we have
\begin{equation}\label{lesecondpetit}
\tilde \bP^n_{N,\gb}\left[\cD^{\cc}_z \ | \ U(\phi)=\gG \right]\le C \binom{M^2}{\kappa} e^{-30 \beta  n} \le h^3/2.
\end{equation}
Hence the first term in \eqref{zeconti} is larger than $-h^3$.
The second term  is larger (by Jensen's inequality) than 
\begin{equation}\label{oneterm}
 \tilde \mu_{z,0}\left(\sum_{x\in \cB'_z} \delta_x\right)\ge (M-2n^2)^2( \theta_1 h^{2}- C h^{5/2})\ge M^2 \left( \theta_1 h^{2} - h^{2+\gep}\right)
\end{equation}
where the first inequality is a consequence of \eqref{toutt}
and the second one follows  if from the fact that $n$ is of order $|\log h|$ and $M= h^{-1/50}$ (assuming that $\gep<1/50$).
Now concerning the second term in \eqref{thebegin}, we note that {\blue using \eqref{touttt}} and our choice of parameter we have
\begin{equation}
  \tilde \mu_{z,h}( e^{\sum_{x\in \cB_z}(\alpha \go_x-\gl(\alpha))})
  \ge \tilde \mu_{z,h}\left(\sum_{x\in \cB_z} \delta_x=0\right)\ge {\blue 1 - C e^{-4\beta n}|\cB_z|} \ge  1-\sqrt{h}.
\end{equation}
Hence using the inequality
$$\log x\ge (x-1)+\frac{1}{2y^2}(x-1)^2$$
valid for all $x\ge y$ we obtain that {\blue
\begin{multline}\label{zebgin}
 \bbE \left[\log  \tilde \mu_{z,h}( e^{\sum_{x\in \cB_z}(\alpha \go_x-\gl(\alpha))}) \right]\\
 \ge
 \bbE \left[ \tilde \mu_{z,h}( e^{\sum_{x\in \cB_z}(\alpha \go_x-\gl(\alpha))}-1)\right]+
 \frac{\bbE \left[ \tilde \mu_{z,h}( e^{\sum_{x\in \cB_z}(\alpha \go_x-\gl(\alpha))}-1)^2\right]}{2(1-\sqrt{h})^2}.
\end{multline}
The first term in the r.h.s.\ equals zero.}
Now the variance term above can be expressed as 
\begin{multline}
 \tilde \mu^{\otimes 2}_{z,h}( e^{[\gl(2\alpha)-2\gl(\alpha)]\sum_{x\in \cB_z} \delta^{(1)}_x\delta^{(2)}_x}-1)\\
 \le (e^{\gl(2\alpha)-2\gl(\alpha)}-1) \tilde \mu^{\otimes 2}_{z,h}\left(\sum_{x\in \cB_z} \delta^{(1)}_x\delta^{(2)}_x=1\right)+e^{\kappa[\gl(2\alpha)-2\gl(\alpha)]}\tilde \mu^{\otimes 2}_{z,h}\left(\sum_{x\in \cB_z} \delta^{(1)}_x\delta^{(2)}_x\ge 2\right).
\end{multline}
Using Markov inequality to estimate the first term we have  then from Lemma \ref{zetims} (more precisely \eqref{toutt}-\eqref{touttt}) 
\begin{multline}\label{zrip}
 \tilde \mu^{\otimes 2}_{z,h}\left(\sum_{x\in \cB_z} \delta^{(1)}_x\delta^{(2)}_x\ge 1\right)\le \tilde \mu^{\otimes 2}_{z,h}\left(\sum_{x\in \cB_z} \delta^{(1)}_x\delta^{(2)}_x\right) \\ \le 
 C|\cB_z\setminus \cB'_z |  e^{-8\gb n}
 + | \cB'_z | (\theta^2_1 e^{-8\gb n}+C h^5/2) \le M^2 ( \theta^2_1 e^{-8\gb n}- h^{2+\gep}).
\end{multline}
To estimate the second term, we combine \eqref{mcdouble} and a union bound  we have 
\begin{multline}\label{zrop}
 \tilde \mu^{\otimes 2}_{z,h}(\sum_{x\in \cB_z} \delta^{(1)}_x\delta^{(2)}_x\ge 2)\le \sum_{x,y\in \cB_z} \mu^{\otimes 2}_{z,h}(\delta^{(1)}_x\delta^{(2)}_x=1 \text{ and } \delta^{(1)}_y\delta^{(2)}_y=1  )\\
 = \sum_{x,y\in \cB_z} \tilde \mu_{z,h}(\delta_x\delta_y)^2\le C M^4 e^{-12\gb n}\le  M^2 h^{5/2}.
\end{multline}
Combining the inequalities \eqref{zebgin}-\eqref{zrop} we obtain
\begin{equation}\label{twoterm}
  \bbE \log  \tilde \mu_{z,h}( e^{\sum_{x\in \cB_z}(\alpha \go_x-\gl(\alpha))}) \ge - M^2  (\theta^2_1 e^{-8\gb n}+ C h^{2+\gep} ).
\end{equation}
Hence from \eqref{thebegin}, \eqref{oneterm} and \eqref{twoterm}, we can conclude that 
\begin{equation}
 \log \tilde \bE^n_{N,\gb}\left[ \ind_{\cD_z}e^{\sum_{x\in \cB_z} (\alpha \go_x- \gl(\alpha)+h)\delta_z}  \ | \ U(\phi)=\gG \right]
 \ge  M^2 (h e^{-4\gb n}-    \theta^2_1 e^{-8\gb n}- C h^{2+\gep} ).
\end{equation}
\qed

\begin{proof}[Proof of Lemma \ref{zetims}]
The three statements are deduced from Proposition \ref{rouxrou}.
For $\gG$ a realization of $U(\phi)$ which includes no contour longer than $n^4$, let us consider
\begin{equation}
\tilde \mu(\cdot):=\tilde \bE^n_{N,\gb}\left[\phi \restrict_{\cB_z}\in \cdot \ | \ U(\phi)=\gG \right] 
\end{equation}
Note that $\tilde \mu$ is of the form $\bP^n_{\bL, \gL,\gb}$ (here $\gL=\cB_z$ and the contour restriction $\bL$ is determined by the boundary condition  $U(\phi)=\gG$ as well as  by the exclusion of long contours).
Recall that by definition any point in $\cB'_z$ is located at a distance from the boundary which is larger than $n^4/2$. Hence  we have for the set  $\bL$ which appears implicitly in the definition of $\tilde \mu$ and any choice of $x\in \cB'_z$ (recall \eqref{deffdiss})
\begin{equation}
 d(x,\bL^{\cc})\ge n^2.
\end{equation}
Hence we can deduce from \eqref{zups}-\eqref{zup2s}-\eqref{zup3s} that 
\begin{equation} \label{leca0}
 \begin{split}
 \tilde \mu ( \delta_x ) &\le C e^{-4\beta n}, \quad \quad  \forall x \in \cB_z 
 \\
 |\tilde \mu ( \delta_x )-\theta_1 e^{-4\beta n}| &\le C e^{-6\beta n}, \quad  \quad \forall x \in \cB'_z,\\
 \tilde \mu ( \delta_x \delta_y)&\le C e^{-6\beta n}, \quad \quad \forall x,y \in \cB_z.
 \end{split}
\end{equation}
To conclude we only need to show that for any event $A$ we have 
\begin{equation}\label{alla}
 |\tilde \mu_{z,h}(A)-\tilde \mu(A)| \le h^3 + 2 M^2 h \mu(A).
\end{equation}
Indeed using \eqref{alla} for the events $\{\delta_x=1\}$ and $\{\delta_x\delta_y=1\}$ and recalling that $M=h^{-1/50}$ we can deduce all the required estimates from  \eqref{leca0}.
Now let us prove \eqref{alla}, we have
\begin{equation}
 \tilde \mu_{z,h}(A)= \frac{\tilde \mu( e^{h\sum_{x\in \cB_z} \delta_x}\ind_{\cD_z\cap A})}{\tilde \mu( e^{h\sum_{x\in \cB_z} \delta_x}\ind_{\cD_z})}=(1+ r(A,h)) \tilde \mu( A \ | \ \cD_z)
\end{equation}
where $r(A,h)$  satisfies
\begin{equation}
 e^{-M^2 h}-1 \le r(A,h)\le e^{M^2 h}-1.
\end{equation}
Thus for $h$ sufficiently small $|r(A,h)|\le (3/2) M^2 h$.
 Hence we can conclude by observing that from \eqref{lesecondpetit} we have
$$|\tilde \mu( A \ | \ \cD_z)-\tilde \mu (A)|\le  h^3/2 .$$

\end{proof}

\section{Upper bound on the free energy}\label{sec:daupper}

\subsection{Result and strategy of proof}

The aim of the section is to prove the following quantitative statement, which together with Proposition \ref{dalower}, completes the proof of Theorem \ref{mainres}.

\begin{proposition}\label{daupper}
There exist positive constants $\beta_0$ and $\gep$ such that for all $\gb\ge \gb_0$ and $\alpha>0$ there exists a constant $C(\alpha,\beta)$ such that for all $h>0$ 
  \begin{equation}\label{zookup}
\bar \tf_{\beta}(\alpha,h) \le  G_{\beta}(\alpha,h)+ C(\alpha,\beta) h^{2+\gep}
 \end{equation}

 \end{proposition}
To explain the proof strategy, we need to   mention a result from \cite{cf:GL3}, which allows to bound the expectation of the log-partition function without using any information about the distribution of the contact set $(\delta_x)_{x\in \gL_N}$.

\begin{proposition} 
\label{th:ub}
{\blue
Consider $\gL$ a finite set, $(\go_x)_{x\in \gL}$ a field of IID  positive random variables with probability  distribution denoted $\bbP$ and  
an arbitrary random vector $(\delta_x)_{x\in \gL}$ on $\{0,1\}^{\gL}$  with probability distribution denoted by $\bP_{\gL}$. }
Then we have 
\begin{equation}\label{thegeneral}
 \bbE \log \bE_{\gL}\left[ e^{\sum_{x\in \gL}(\alpha \go_x-\gl(\alpha))\delta_x} \right]
 \le |\gL|\max_{p\in[0,1]} \bbE \left[ \log \left(1+p(e^h \xi-1) \right) \right]
\end{equation}
where $\xi:=e^{\alpha \go-\gl(\alpha)}$.
\end{proposition}
The result can simply be obtained by induction on the number of vertices $|\gL|$.
Let us set  
$\xi_x:=e^{\alpha \go_x-\gl(\alpha)}$.
The above result indicates in particular that the expectation is maximized when
$(\delta_x)_{x\in \gL}$ is a field of IID Bernoulli random variables with a parameter which maximizes the left-hand side of \eqref{thegeneral}.
Using Proposition \ref{th:ub} for our partition function we obtain immediately that 
\begin{equation}\label{dopez}
\bar \tf_{\beta}(\alpha,h)\le \max_{p\in[0,1]} \bbE \left[ \log \left(1+p(e^h \xi-1) \right) \right].
\end{equation}

\medskip
{\blue
The bound \eqref{dopez}, while not sharp is going to help us to provide the intuition for the proof of Proposition \ref{daupper}. Let us develop on this point.
If one fixes the boundary condition equal to some large $n$, the distribution of $(\ind_{\{\phi(x)=0\}})_{x\in \gL_N}$ is a good approximation of a Bernoulli product measure with parameter   $p_n=\theta_1 e^{-4\gb n}$ (cf. Proposition  \ref{rouxrou}). If one temporarily accepts that the best optimization one can make is select the optimal boundary condition (like we have done for the lower bound), it looks plausible that the maximum in \eqref{dopez} should not be taken over all values of $p$ but only along the sequence $p_n$ yielding the following approximation for the free energy
$$ \max_{n\ge 1} \bbE \left[ \log \left(1+\theta_1 e^{-4\gb n}(e^h \xi-1) \right) \right].$$
This latter quantity is asymptotically equivalent to $G_{\gb}(\alpha,h)$ (cf.\ Section \ref{ozermodel}).

\medskip

 To make this reasoning rigorous, we are going to show that after conditioning to the realization of  \textit{large} contours (rigorously defined below), in most of the box $\gL_N$, the process $(\ind_{\{\phi(x)=0\}})_{x\in \gL_N}$  is locally well approximated by IID Bernoulli variables of parameter $p_n$ for some value of $n$.} More precisely, we show that large contours occupy a small fraction of the  space and that in regions where there are no large contour, 
the Bernoulli field is a good approximation.
 We proceed as follows:

\begin{itemize}
 \item Firstly, we make a decomposition of the partition function based on the  realization of the set of large contours. In order to use some convenient sub-additivity property we 
 replace $\bbE \log Z_N$ by $ \frac{1}{\theta} \log \bbE Z^{\theta}_N$ for some $\theta\in (0,1)$. By choosing $\theta$ very close to zero,
 we can make this change having a negligible effect.
 \item Then, after conditioning on the realization of the set of large contours, we divide $\gL_N$ into large cells, an by using the information given by Proposition \ref{rouxrou} concerning the contact density in each cell, we manage to obtain a sharp bound for the contribution of each cell to the $\log$-partition function.
\end{itemize}
We need to fix a few parameters. There is a wide range of valid choices, but the following one turns out to be convenient in the computations.
Given $\gep>0$ we fix (taking integer parts when necessary)
\begin{equation}\label{parametzer}
 \theta:= h^{\gep}, \quad  L:= h^{-2\gep}, \quad M:=L^2.
\end{equation}
We assume that $N$ is an integer multiple of $M$ (and thus of $L$).
Our first observation is that 
\begin{equation}\label{jenson}
 \bbE\left[ \log  Z^{h,\alpha,\go}_{N,\gb}\right]\le \frac{1}{\theta} \log  \bbE \left[ \left( Z^{h,\alpha,\go}_{N,\gb}\right)^{\theta}\right]
\end{equation} 
For $A\subset \gO_N$  we let $
Z^{h,\alpha,\go}_{N,\gb}(A)$ denote the reduced partition function restricted to $A$
$$ Z^{h,\alpha,\go}_{N,\gb}(A):=\bE_{N,\beta}\left[ e^{\sum_{x\in \gL_N}(\alpha \go_x-\gl(\alpha)+h)} \ind_{A}\right].$$
If $(A_i)_{i\in \cI}$ is a partition of $\gO_N$,
then using the inequality
 $(\sum_{i\in I} a_i)^{\theta}\le \sum_{i\in I}  a^{\theta}_i$ valid for any collection of positive real numbers and any $\theta\in (0,1)$,  we have
\begin{equation}\label{labeel}
 \left( Z^{h,\alpha,\go}_{N,\gb}\right)^{\theta}\le \sum_{i \in \cI} \left(Z^{h,\alpha,\go}_{N,\gb}(A_i)\right)^{\theta}.
\end{equation}
We partition $\gO_N$ according to the realizations of the set of large contours. We say that a contour is large if $|\tilde \gamma|\ge L$ and we let 
$\Upsilon^{\larg}(\phi)$ denote the set of large contours in $\Upsilon(\phi)$). For the inequality \eqref{labeel} to be not too far off we want 
 the cardinality of the set $\cI$ to be as small as possible. 
 To this end, we group realizations of  $\Upsilon^{\larg}(\phi)$ which are in a sense close to one another to include them in  a common event. This brings the necessity of introducing a \textit{coarse grained} version of contours on the scale $L$.

\subsection{Coarse grained large contour} \label{seccg}

We divide $\bbR^2$ into boxes of side-length $L$ by setting
$B_z:=[0,L)^2+Lz$. Recalling that 
\begin{equation}\label{deflarge}
 \Upsilon^{\larg}(\phi):=\{ \gamma \in \Upsilon(\phi) \ : \ |\tilde \gamma|\ge L \}.
\end{equation}
We wish to define coarse grained versions of the contours at scale $L$.
We start by defining  the coarse-grained trace of a contour 
\begin{equation}\label{cgtrace}
 \chi(\gamma):=\{ z\in \bbZ^2 \ : \ \tilde \gamma \cap  B_z\ne \emptyset \}
\end{equation}
where $\tilde \gamma$ above is identified with the union of the geometric segments associated with its edges.
Note that $\chi(\gamma)$ is a connected subset of $\bbZ^2$.
We also define $\Int_\gamma: \bbZ^2\to \{0,1/2,1\}$, {\blue the coarse grained \textit{interior function}  which is the coarse grained approximation of  the function $\ind_{\bar \gamma}$}
\begin{equation}
 \Int_\gamma(z):=\begin{cases}
                0 \text{ if } B_z \subset \bar \gamma^{\cc},\\
                1/2 \text{ if } z\in \chi(\gamma),\\
                1 \text{ if } B_z \subset \bar \gamma.
            \end{cases}
\end{equation}
An important observation is that $\Int_{\gamma}$ is not determined by $\chi(\gamma)$ (see Figure \ref{okay}). However, the definition implies that $\Int_{\gamma}$ is constant on all connected components of $\bbZ^2\setminus \chi$. Hence there are at most (and the counting is rough) $2^{|\chi|}$ possible interior functions corresponding to a given trace (this is because  $\bbZ^2\setminus \chi(\gamma)$ has at most $|\chi|$ connected components).

 \begin{figure}[ht]
\begin{center}
\leavevmode
\epsfxsize = 13 cm
\psfrag{L}{$L$}
\psfrag{r}{$L$}
\epsfbox{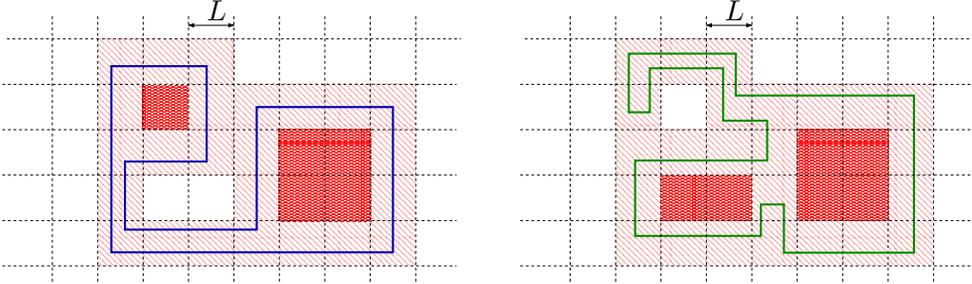}
\end{center}
\caption{\label{okay} 
Schematic representation of two contours with coarse grained trace and interior (the hatched part correspond to the value $1/2$ and the darker one to the value one). Note that while the two contours have the same trace, they have different interior functions.}
\end{figure}

\medskip

A coarse grained cylinder $\hat \chi$ is a triplet $(\chi, \iota, q)$ made of a connected subset of $\bbZ^2$, a function $\bbZ^2\to \{0,1/2,1\}$ which is constant on the connected components of $\bbZ^2 \setminus \chi$ and an intensity $q\in \bbZ$.
We say that $(\chi, \iota, q)$ is a coarse grained cylinder for $\phi$ if 
\begin{itemize}
 \item[(A)] There exists at least one contour $\gamma\in \Upsilon^{\larg} (\phi)$ such that $\chi(\gamma)=\chi$ and $\Int_{\gamma}=\Int$.
 \item[(B)] We have {\blue(recall that by convention if $\hat \gamma$ is a cylinder, $\gamma$ denotes the associated contour)}
 \begin{equation}
 \sum_{\{ \hat \gamma \in \hat \Upsilon^{\larg}(\phi) \ : \ \chi(\gamma)=\chi \text{ and  } \Int_{\gamma}=\iota \}} k(\hat \gamma)=q,
 \end{equation}
\end{itemize}
When $q\ne 0$, the condition (B) implies of course that (A) holds.
We let $\Xi(\phi)$ denote the set of coarse grained cylinders associated with $\phi$. 
{\blue We let
$\mathfrak{Z}_{N}$ denote the set of coarse grained cylinders that can be obtained from contours in $\gL_N$ (the set $\mathfrak{Z}_{N}$ also depends on $h$ via $L$ but we omit this dependence for better readability).  For $\mathfrak{C}\subset \mathfrak{Z}_{N}$
we introduce the event $B_{\kC}:=\{ \Xi(\phi)= \kC \}$. Note that we have not imposed any compatibility condition so that for many instances of $\kC$, $B_{\kC}$ is an empty event.
We consider 
 $\mathfrak{D}_N^{\mathrm{ld}}$ the set of ``low density'' coarse grained cylinder collections
\begin{equation}\label{ddk}
 \mathfrak{D}_N^{\mathrm{ld}}:=\{  \mathfrak{C}\subset \mathfrak{Z}_{N} \ : \ \sum_{(\chi,\iota, q) \in \mathfrak{C}}|\chi|\le \sqrt{h}N^2  \},
\end{equation}
and let  $\mathfrak{D}_N^{\mathrm{hd}}$ denote its complement in the power set $\cP(\mathfrak{Z}_{N})$.}
We set $\bar \cB:= \bigcup_{\kC\in \mathfrak{D}_N^{\mathrm{hd}}}\cB_{\kC}$.
Using \eqref{labeel} we have
\begin{multline}\label{labeelo}
 \bbE \left[\left( Z^{h,\alpha,\go}_{N,\gb}\right)^{\theta}\right]\le \sum_{\kC \in \mathfrak{D}_N^{\mathrm{ld}}} \bbE \left[\left(Z^{h,\alpha,\go}_{N,\gb}(\cB_\kC)\right)^{\theta}\right]+\bbE \left[\left(Z^{h,\alpha,\go}_{N,\gb}(\bar \cB)\right)^\theta\right]\\
 \le  \left(\sum_{\kC \in \mathfrak{D}_N^{\mathrm{ld}}} \bP_{N,\gb}[\cB_\kC]^{\theta}\right)
 \max_{\kC \in \mathfrak{D}_N^{\mathrm{ld}}} \bbE \left[ Z^{\theta}_\kC\right]+\bbE \left[\left(Z^{h,\alpha,\go}_{N,\gb}(\bar \cB)\right)^\theta\right].
 \end{multline}
 where
 \begin{equation}
 Z_{\kC} := \bE_{N,\gb}\left[ e^{\sum_{x\in \gL_N}(\alpha \go_x-\gl(\alpha)+h)\delta_x} \ | \ \cB_{\kC}\right].
\end{equation}
Let us first show that the second term in the r.h.s.\ of  \eqref{labeelo} is small. Using the annealed bound we have 
\begin{equation}
 \bbE \left[\left(Z^{h,\alpha,\go}_{N,\gb}(\bar \cB)\right)^\theta
 \right]\le  \left(Z^{h}_{N,\gb}(\bar \cB)\right)^\theta \le 
 (e^{N^2 h} \bP_{N,\beta}[ \bar \cB])^{\theta}.
\end{equation}
We can then use the following rough estimate (proved at the end of the section) which implies that the second term in the r.h.s.\ of \eqref{labeelo} is smaller than $1$.
\begin{lemma}\label{grancontour}

For $\beta$ sufficiently large and $h$ sufficiently small
we have
\begin{equation}
\bP_{N,\beta}[ \bar \cB]\le e^{- \sqrt{h} N^2/2}.
\end{equation}
\end{lemma}
Considering now the first term in the r.h.s.\ of \eqref{labeelo},
the factor $\left(\sum_{\kC \in \mathfrak{D}_N^{\mathrm{ld}}} \bP_{N,\gb}[\cB_\kC]^{\theta}\right)$ which materializes  how much we might have lost in our coarse graining decomposition can also be controlled with the rough estimates exposed in the following lemma, whose proof is postponed to the next subsection.
\begin{lemma}\label{commenfer}
 For every  $\mathfrak{C}\subset \mathfrak{Z}_{N}$ we have
 \begin{equation}\label{lezcontz}
  \bP_{N,\gb}[ \Xi(\phi)=\cB_\kC ] \le e^{-\frac{\gb L}{10} |\tilde \kC| }.
 \end{equation}
 where  $|\tilde \kC|$ is defined by $$|\tilde \kC|:= \sum_{ (\chi,\iota,q) \in \kC} 
 \max(|q|,1) |\chi|.$$
 As a consequence we have 
 \begin{equation}\label{lota}
 \sum_{\kC \in \mathfrak{D}_N^{\mathrm{ld}}} \bP_{N,\gb}[\cB_\kC]^{\theta}\le 
e^{N^2 h^3}.
\end{equation}
 \end{lemma}
\noindent The most delicate part is to control the value of $\bbE \left[ Z^{\theta}_\kC\right]$ for all $\mathfrak{C}\subset \mathfrak{Z}_{N}$.

\begin{proposition}\label{principalz}
 There exists positive $\beta_0$ and $h_0$ such that for $\beta \ge \gb_0$, and $h\in (0,h_0]$  we have for every $\mathfrak{C}\subset \mathfrak{Z}_{N,L}$,
 \begin{equation}
 \frac{1}{\theta} \log \bbE \left[ Z^{\theta}_{\kC}\right]
  \le N^2 \left(G_{\gb}(\alpha,h)+h^{2+\gep}\right)
 \end{equation}

 \end{proposition}
The proof of Proposition \ref{principalz} runs from   Section \ref{pprincp} to \ref{ccc}.
Note that as a consequence of Lemma \ref{grancontour}, Lemma \ref{commenfer} and Proposition \ref{principalz} we deduce from \eqref{labeelo} that for small $h$ we have 
\begin{equation}
  \bbE \left[\left( Z^{h,\alpha,\go}_{N,\gb}\right)^{\theta}\right]
  \le 1 + e^{ N^2 (G_{\gb}(\alpha,h)+h^3+h^{2+\gep})}.
\end{equation}
which is sufficient to deduce Proposition \ref{daupper} (recall \eqref{jenson}).
\qed
 
\subsection{Proof of Lemma \ref{grancontour}}

While the result could be deduced from Lemma \ref{commenfer}, it is substantially simpler to replicate the argument given in the proof  \cite[Lemma 8.3]{cf:part1}.  
  
  \medskip
  
  If $\phi \in \bar \cB$ then the sum of the length of its large contours has to be large. 
  Indeed for every coarse grained cylinders $(\chi,\iota,q)$ of $\phi$, there must be at least one contour $\gamma \in \Upsilon^{\larg}(\phi)$ which satisfies $\chi(\gamma)=\chi$, in particular the length of this contour must be larger than $|\chi|$ (in fact it must be larger than $(L/9)|\chi|$ but this is of no importance here). Hence the sum of the length of large contours must satisfy
  \begin{equation}
   \cL^{\larg}(\phi):=\sum_{\gamma \in \Upsilon^{\larg}(\phi)} |\tilde \gamma| \ge \sum_{(\chi, \iota, q)\in \Xi(\phi)} |\chi| \ge 
  \sqrt h  N^2.
  \end{equation}
Now we can follow the proof of \cite[Lemma 8.3]{cf:part1} which helps to control the Laplace transform of $\cL^{\larg}$.
Using the stochastic domination of Lemma \ref{restrict}, we have for $\gl\ge 0$ ({\blue the quantity $e^{\gl \cL^{\larg}(\phi)}$ is an increasing function of the set of contour $\Upsilon(\phi)$})
\begin{equation}
 \bE_{N,\gb} \left[ e^{\gl \cL^{\larg}(\phi)}\right]
 \le \prod_{\{ \gamma \in \cC_{\gL_N} \ : \ |\tilde \gamma| \ge L \}}
 (1+e^{-\beta |\tilde \gamma|}(e^{\gl|\tilde \gamma|}-1)).
\end{equation}
Only adding extra-factors which are all larger than one in the product we observe that 
$\prod_{\{ \gamma \in \cC_{\gL_N} \ : \ |\tilde \gamma| \ge L \}} \dots \le 
\prod_{x\in \gL_N}\prod_{\{\gamma\in \cC \ : \ x\in \bar \gamma \ , \ |\tilde \gamma| \ge L \} } \dots$, and by translation invariance we have
\begin{equation}
  \bE_{N,\gb} \left[ e^{\gl \cL^{\larg}(\phi)}\right]\le \left( \prod_{\{\gamma\in \cC \ : \ {\bf 0}\in \bar \gamma \ , \ |\tilde \gamma| \ge L \} }  (1+e^{-\beta |\tilde \gamma|}(e^{\gl|\tilde \gamma|}-1)) \right)^{N^2}
\end{equation}
Using that $\log (1+t)\le t$ and observing that there are at most $n4^n$ contour of length $n$ which goes around $0$, this yields in turn
\begin{equation}
 \frac{1}{N^2}  \log  \bE_{N,\gb} \left[ e^{\gl \cL^{\larg}(\phi)}\right]
 \le \sum_{\{\gamma\in \cC \ : \ {\bf 0}\in \bar \gamma \ , \ |\tilde \gamma| \ge L \} }   e^{-\beta |\tilde \gamma|}(e^{\gl|\tilde \gamma|}-1) \le \sum_{n\ge L} n 4^n e^{(\gl-\beta) n}.
\end{equation}
If $\beta \ge 3$ and $\gl= 1$, we obtain that the right hand side is smaller than $e^{-L/10}$ (which is smaller than any powers of $h$). 
Then we can conclude using the fact that 
\begin{equation}
 \bP_{N,\beta}[  \cL^{\larg}(\phi) \ge  \sqrt{h} N^2] \le e^{-\sqrt{h} N^2}  \bE_{N,\beta}  \left[  e^{\cL^{\larg}(\phi)}\right].
\end{equation}

\qed

 \subsection{Proof of Lemma \ref{commenfer}}
 
For notational simplicity we prove  a bound for the probability of the presence of a given coarse-grained cylinders 
 \begin{equation}\label{onecylinder}
  \bP_{N,\gb}\left[ (\chi,\iota,q)\in \Xi(\phi) \right]\le e^{-\frac{\gb L \max(|q|,1)|\chi|}{10}}.
 \end{equation}
 The reader can then check that the same proof yields for any finite family of coarse grained cylinders $(\chi_i,\iota_i,q_i)_{i\in \cI}$
  \begin{equation}\label{severalcylinder}
  \bP_{N,\gb}\left[ (\chi_i,\iota_i,q_i)\in \Xi(\phi) , i \in \mathcal I \right]\le e^{-\frac{\gb L \sum_{i\in \cI }\max(|q_i|,1)|\chi_i|}{10}},
 \end{equation}
 which is a stronger statement than \eqref{lezcontz}.
Given  positive integers $m\ge 1$ $k_1, k_2\dots, k_m\ge 1$   and $\gep_1,\dots,\gep_m\in
\{-1,1\}$ and  $\iota$ a  function $\bbZ^2\to \{0,1/2,1\}$, we are going to show that 
\begin{multline}\label{stouf}
 \bP_{N,\gb}\Big[ \exists \gamma_1,\dots, \gamma_m\in \Upsilon^{\larg}(\phi) \\
 \forall i \in \lint 1, m \rint, \gep(\gamma_i)=\gep_i, k(\gamma_i)=k_i \text{ and }\Int_{\gamma_i}=\iota  \Big]
 \le e^{-\frac{\gb L \sum_{i=1}^m k_i|\chi|}{8}}.
 \end{multline}
In order to deduce \eqref{onecylinder} from \eqref{stouf}, we use a union bound and sum over all  possible $m$, $k_i$ and $\gep_i$ which are such that
$\sum_{i=1}^m \gep_i k_i=q$ (excluding of course $m=0$ when $q=0$)
Note that we must have necessarily $\sum_{i=1}
^m k_i\ge |q|\vee 1$. 
Hence slightly overcounting,  we have 

\begin{equation}
   \bP_{N,\gb}\left[ (\chi,\iota,q)\in \Xi(\phi) \right]
   \le \sum_{q'\ge |q|\vee 1} \sum_{m\ge 1} \sum_{ \{ {\bf k} \ : \sum_{i=1}^{m}k_i=q' \}} 2^m e^{-\frac{\gb L q'}{8}}.
 \end{equation}
Note that given $q'$ there are exactly $2^{q'-1}$ possibilities to chose 
the pair $(m,{\bf k})$ and hence the above sum is smaller than
\begin{equation}
 \sum_{q'\ge |q|\vee 1} 4^{q'}e^{-\frac{\gb L q'}{8}},
\end{equation}
which is sufficient to conclude the proof if $L$ is sufficiently large.

Now let us prove \eqref{stouf}.
Let $z_{\min}$ be  the smallest element in $\chi$ for the lexicographical order. We observe now that a contour which satisfies $\Int_{\gamma}=\iota$ must have length at least 
$\ell_{\min}:=L\min(1, |\chi|/4)\ge L |\chi|/4$ (simply because $L$ step are needed to visit more than $4$ squares) and a point in $B_{z_{\min}}$.
For $\ell\ge \ell_{\min}$ there are thus  at most $4^\ell L^2$ possible geometric contours of length $\ell$.
Using a greedy union bound (choosing each of the contours independently and summing over all possible contours) we obtain when $\beta \ge 3$  the l.h.s.\ of \eqref{stouf} is smaller than 

\begin{multline}
 (L^2)^m\sum_{\ell_1,\dots,\ell_m \ge \ell_{\min}} 4^{\sum_{i=1}^m \ell_i }  e^{-\gb \sum_{i=1}^m \ell_i k_i} 
 \le (2 L^2)^m 4^{m \ell_{\min}}e^{- \gb|\chi| \ell_{\min}\sum_{i=1}^m k_i} \\ \le e^{-\frac{\gb L |\chi| \sum_{i=1}^m k_i}{8}}.
\end{multline}
To prove \eqref{lota} we observe that as a consequence of \eqref{lezcontz} we have
\begin{equation}
 \sum_{\kC \in \mathfrak{D}_N^{\mathrm{ld}}} \bP_{N,\beta}[\cB_{\kC}]^{\theta}\le  \sum_{\kC \subset \mathfrak Z_N} e^{-\frac{\gb L |\tilde \kC|\theta}{10}}= \prod_{(\chi,\iota,q)\in \mathfrak Z_N}\left(1+ e^{-\frac{\gb L |\tilde \kC|\theta}{10}}\right).
\end{equation}
Taking the logarithm we obtain that
\begin{equation}\label{yuy}
  \log \left( \sum_{\kC \in \mathfrak{D}_N^{\mathrm{ld}}} \bP_{N,\beta}[\cB_{\kC}]^{\theta}\right)
  \le \sum_{(\chi,\iota,q)\in \mathfrak Z_N }e^{-\frac{\gb L (|q|\vee 1) |\chi|\theta}{10}}.
\end{equation}{\blue
Let us first evaluate the sum over $q$ and $\iota$ for a fixed $\chi$.
We have (provided $L\theta=h^{-\gep}$ is sufficiently large)
\begin{equation}\label{sumonq}
 \sum_{q \in \bbZ}e^{-\frac{\gb L (|q|\vee 1) |\chi|\theta}{10}}\le 4e^{-\frac{\gb L  |\chi|\theta}{10}}.
\end{equation}}
Then there are at most $2^{|\chi|}$ interior functions corresponding to one coarse grained contour $\chi$.
Finally  the number of choice for a connected set $\chi$ with $|\chi|=k$ is at most $(N/L)^2 16^{k}$.  Indeed for every such set $S$, there exists a lattice path of length $2k-1$ whose range is exactly $S$. Then  $(N/L)^2$ is the number of options for the starting point and $16^k$ is an upper bound on the number of lattice paths. {\blue Combining the above observation we obtain that 
\begin{equation}
 \sum_{(\chi,\iota,q)\in \mathfrak Z_N }e^{-\frac{\gb L (|q|\vee 1) |\chi|\theta}{10}}
  \le \sum_{k\ge 1} 2^k  (N/L)^2 16^{k}   e^{-\frac{\gb L k\theta}{10}}\le N^2 e^{-\frac{\gb L  \theta}{10}}\le N^2 h^3/4.
\end{equation}
where the two last inequalities are valid if $\theta L= h^{-\gep}$ is sufficiently large, which in view of \eqref{yuy} is sufficient to conclude.}

\subsection{Proof of Proposition \ref{principalz}}\label{pprincp}

Now we want to make a coarser decomposition of our lattice.
We split $\gL_N$ into cells of size $M$ (recall that $M=L^2$.
More precisely we set for $y\in \lint 0, (N/M)-1\rint^2$ ({\blue see Figure \ref{ccy}})
$$\cC_y:= yM + \lint 2L+1, M-2L \rint^2 \text{ and } \bar \cC_y:= yM + [ 0, M )^2.$$

\begin{figure}[ht]
\begin{center}
\leavevmode
\epsfysize = 4 cm
\psfrag{cy}{$\mathcal C_y$}
\psfrag{barcy}{$\bar{\mathcal C}_y$}
\epsfbox{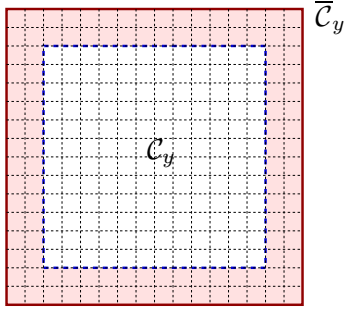}
\end{center}
\caption{\label{ccy} {\blue The  cell $\cC_y$ (enclosed by the dotted blue  line) and its slightly larger version $\bar { \cC}_y$ (enclosed in the solid red line). These cells are much larger than the cells $B_z$ (represented with thinner dotted lines, in the picture $L=4$). As the boundary of $\cC_y$ is at a distance $2L$ from that of $\bar{\cC_y}$, the behavior of $\phi$ inside $\cC_y$ is not much affected by boundary  effects after conditioning to the realization of $\phi$ outside $\bar{\cC}_y$, provided that $\cC_y$ is a good cell.}}
\end{figure}

Given $\kC$ a set of coarse grained contour, say that $\cC_y$ is a \textit{good cell} {\blue for $\kC$} if there are no coarse grained contours which intersects it. Or more precisely if 
$$ \forall (\chi,\iota,q)\in \kC, \quad  \forall z\in \chi, \quad B_z \cap  \bar \cC_y =\emptyset.$$
Our assumption  that our set of coarse grained contours has low density $\kC\in \mathfrak D^{\mathrm{ld}}_N$ ensures that there are at most $N^2  \sqrt{h}$ bad cells and thus that the proportion of bad cells is at most of order  $h^{\frac{1}{2}-8\gep}$ (recall \eqref{ddk}, we assume $\gep\le 1/20$).
We let $\cG_N(\kC)$ (good sites) and $\cR_N(\kC)$ (the rest) denote respectively the set of lattice sites which are resp.\ are not in  a good cell 
\begin{equation}
 \begin{split}
  \cG_N(\kC)&:= \bigcup_{\{y \ : \ \cC_y \text{ is good } \}} \cC_y,\\  \cR_N(\kC)&:=  \gL_N \setminus \cG_N(\kC).
 \end{split}
\end{equation}
From the definition of $\mathfrak D^{\mathrm{ld}}_N$, there exists a positive constant $C$ such that
\begin{equation}\label{boundingtherest}
|\cR_N(\kC)|\le C \left( N^2 h^{2\gep}+ N^2 h^{\frac{1}{2}-8\gep} \right) \le 2C h^{2\gep}.
\end{equation}
where the first term takes into account sites which are in $\bar \cC_y\setminus \cC_y$ and the second one, sites that are in bad cells. The second inequality comes from $\gep \le 1/20$.
We can write 
\begin{equation}\label{decozh}
Z_{\kC} = \bE_{N,\gb}\left[ e^{\sum_{x\in  \cG_N(\kC)}(\alpha \go_x-\gl(\alpha)+h)\delta_x} \ | \ \cB_{\kC}\right]
\mu^{\kC,\go}_N\left[ e^{\sum_{x\in\cR_N(\kC)}(\alpha \go_x-\gl(\alpha)+h)\delta_x} \right]
\end{equation}
Where $\mu^{\kC,\go}_N$ is defined by 
$$ \mu^{\kC,\go}_N(A):= \frac{\bE_{N,\gb}\left[\ind_A  \ e^{\sum_{x\in  \cG_N(\kC)}(\alpha \go_x-\gl(\alpha)+h)\delta_x} \ | \ \cB_{\kC}\right]}{\bE_{N,\gb}\left[e^{\sum_{x\in  \cG_N(\kC)}(\alpha \go_x-\gl(\alpha)+h)\delta_x} \ | \ \cB_{\kC}\right]}$$
Our first task, is to ensure that the sites  $x\in\cR_N$ do not yield a contribution larger than $h^2$ to the free energy.
We let $\bbE_{\cR_N}$ denote the expectation with respect to $(\go_x)_{x\in \cR_N}$.

\begin{lemma}\label{techos}
There exists two constant $C_{\alpha}$ and $h_0$ such that for all $h\le h_0$ and $\kC\in $
 
 \begin{equation}\label{aproov}
  \frac 1 {\theta}\log  \bbE_{\cR_N}\left[ \left(\mu^{\kC,\go}_N\left[ e^{\sum_{x\in\cR_N(\kC)}(\alpha \go_x-\gl(\alpha)+h)\delta_x} \right]\right)^{\theta}\right] \le C_{\alpha} |\cR_N(\kC)| h^2.
 \end{equation}
 \end{lemma}
The second task, which is more delicate it to show that each good site yields (at most) a contribution of order $G_{\gb}(\alpha,h)$ to the free energy plus a smaller order correction.

\begin{lemma}\label{different}
There exists $ h_0(\alpha,\beta)$ such that for $h\le h_0(\alpha,\beta)$, for any fixed $\kC\in $ we have
\begin{equation}\label{rantt}
 \frac{1}{\theta}\log \bbE \left[\bE_{N,\gb}\left[ e^{\sum_{x\in  \cG_N(\kC)}(\alpha \go_x-\gl(\alpha)+h)\delta_x} \ | \ \cB_{\kC}\right]^{\theta}
 \right]\le |\cG_N(\kC)|(G_{\gb}(\alpha,h)+h^{2+\gep}).
\end{equation}

\end{lemma}

The proof of these two lemmas are detailed in Section \ref{prevt} and \ref{prevd} respectively. Let us show how this is sufficient to conclude the proof.
Remembering \eqref{decozh} and applying successively Lemma \ref{techos} and \ref{different} we obtain that 
\begin{multline}
 \frac{1}{\theta} \log \bbE[ Z_{\kC}^{\theta} ]
 \le  C_\alpha h^2 |\cR_N(\kC)|+ |\cG_N(\kC)|(G_{\gb}(\alpha,h)+h^{2+\gep})\\
 \le N^2 (G_{\gb}(\alpha,h)+h^{2+\gep})+  C' N^2  h^{2+2\gep}.
\end{multline}
where in the second line we have used \eqref{boundingtherest} (valid for $\kC\in \mathfrak D^{\mathrm{ld}}_N$).
This concludes the proof.

\qed

\subsection{Proof of Lemma \ref{techos}}\label{prevt}

To prove this inequality, 
we apply a result analogous  to Proposition \ref{th:ub} but valid for fractional moments. The proof goes by induction on the number of sites and is given in full details in \cite{cf:GL3}.

 \begin{lemma}\cite[Lemma.\ 5.6]{cf:GL3}
\label{th:ubfrac}
{\blue
Consider $\gL$ a finite set, $(\go_x)_{x\in \gL}$ a field of IID  positive random variables with probability  distribution denoted $\bbP$ and  
an arbitrary random vector $(\delta_x)_{x\in \gL}$ on $\{0,1\}^{\gL}$  with probability distribution denoted by $\bP_{\gL}$. }
Then we have for any $h>0$,
\begin{equation}\label{woox}
  \log \bbE \left( \bE_{\gL}\left[ e^{\sum_{x\in \gL}(\alpha \go_x-\gl(\alpha))\delta_x} \right]^{\theta}\right)
 \le |\gL|\max_{p\in[0,1]}  \log \bbE \left[ \left(1+p(e^h \xi-1) \right)^{\theta} \right].
\end{equation}
where $\xi:=e^{\alpha \go-\gl(\alpha)}$.
\end{lemma}

We apply  Lemma  \ref{th:ubfrac} to the distribution
$\mu^{\kC,\go}_N$ with $\gL:= \cR_N$.
This is a random distribution but it is independent of $(\go_x)_{x\in \cR_N}$ which is all that matters.
Then we simply need to show that 
$$\max_{p\in[0,1]} \frac{1}{\theta}\log \bbE \left[ \left(1+p(e^h \xi-1) \right)^{\theta} \right] \le C_\alpha h^2.$$
In order to work with $\theta$ that does not depend on $h$, we notice that by Jensen inequality we have (for $\theta<1/2$)
$$\bbE \left[ \left(1+p(e^h \xi-1) \right)^{\theta} \right]\le \bbE \left[ \sqrt{1+p(e^h \xi-1)} \right]^{2\theta}.$$
Thus to conclude we just need to prove that
\begin{equation}\label{scroutch} \max_{p\in [0,1]}\bbE \left[ \sqrt{1+p(e^h \xi-1)} \right]-1\le C'_{\alpha} h^2.
\end{equation}
 The function $p\mapsto \bbE \left[ \sqrt{1+p(e^h \xi-1)} \right]$ is stricly concave. Let $p_h$ be where the function attains its unique maximum. 
 Observe that $p_h$ tends $ 0$, when $h$ tends to $0$.
 Indeed as $\bbE \left[ \sqrt{1+p_h(e^h \xi-1)} \right]\ge 1$, {\blue any limit point  $p_*$   of the sequence $(p_h)_{h\ge 0}$ at $0$  must satisfy 
  $$\bbE \left[ \sqrt{1+p_*(\xi-1)} \right]\ge 1.$$
On the other hand if one assumes $p_*>0$, since $\bbE[\xi]=1$ and $\xi$ is not a constant, by  Jensen's (strict) inequality we 
have  
$$\bbE \left[ \sqrt{1+p_*(\xi-1)} \right]<  \sqrt{ \bbE \left[1+p_*(\xi-1)\right]}=1,$$
yielding a contradiction.}
Using a Taylor expansion at third order we have 
\begin{multline}
 \bbE \left[ \sqrt{1+p_h(e^h \xi-1)} \right]-1\\
 =\frac{p_h}{2}\bbE[(e^h \xi-1)]-\frac{p_h^2}{8} \bbE \left[ (e^h \xi-1)^2 \right]+ O(p^3_h)\\
 = \frac{p_h h } {2} -\frac{p_h^2}{8} \Var (\xi) +O(p^3_h+ hp_h^2+h^2 p_h),
\end{multline}
where $O(f_h)$ here denotes here  a quantity smaller in absolute value that $C f_h$ uniformly in $h\in [0,1]$.
It then follows that $p_h\sim \frac{2h}{\Var (\xi)}$ and thus that 

\begin{equation}
  \bbE \left[ \sqrt{1+p_h(e^h \xi-1)} \right]-1= \frac{1}{2 \Var (\xi)}h^2 +O(h^3).
\end{equation}
 
\qed

\subsection{Proof of Lemma \ref{different}}\label{prevd}

The starting point of the proof is that if $\kC$ is fixed, then inside good cells $\cC_y$, the distribution of $\phi$ looks like the infinite volume SOS model, with a bulk height $n_y$ {\blue which is fully determined by $\kC$ and $y$.}

\medskip

Then the idea is that if $n_y$ is large, we can (at the cost of a small error term) replace $(\delta_x)_{x\in \cC_y}$ by IID Bernoulli variables with parameter $\theta_1 e^{-4\beta n_y}$ which yields at most a contribution of order $G_{\beta}(\alpha,h)$ per site (the bound being sharp if $n_y$ takes the optimal value). If $n_y$ is small, the system displays a lot of contacts and a rougher argument is sufficient to show {\blue that the corresponding contribution} to the free energy is negative.

\medskip

Let us go more into details and explain how we perform our comparison with the Bernoulli measure.
The starting point is the following application of H\"older's inequality. For any positive function $g$ of the environment we have
\begin{multline}\label{kolder}
 \bbE\left[\bE_{N,\gb}\left[ e^{\sum_{x\in  \cG_N(\kC)}(\alpha \go_x-\gl(\alpha)+h)\delta_x} \ | \ \cB_{\kC}\right]^{\theta}
 \right] \\
 \le \bE_{N,\gb} \left[\bbE\left[g(\go)^{\theta-1} e^{\sum_{x\in  \cG_N(\kC)}(\alpha \go_x-\gl(\alpha)+h)\delta_x} \right] \ | \ \cB_{\kC}\right]^{\theta}\bbE\left[g(\go)^{\theta}\right]^{1-\theta}.
\end{multline}
Hence the l.h.s.\ of \eqref{rantt} which we have to bound is smaller than

\begin{equation}\label{lolder}
 \frac{1-\theta}{\theta}\bbE\left[g(\go)^{\theta}\right]+ 
  \log \bE_{N,\gb} \left[\bbE\left[g(\go)^{\theta-1} e^{\sum_{x\in  \cG_N(\kC)}(\alpha \go_x-\gl(\alpha)+h)\delta_x} \right] \ | \ \cB_{\kC}\right].
\end{equation}

Our idea is to choose $g$ to be equal to (something close to) the partition function associated with the IID Bernoulli fields which appears in our heuristics.
Let us first define the height of each good cell as follows.
Given $x\in \cG_N(\kC)$ we let $z=z(x)$ be such that  $x\in B_z$. We set
\begin{equation}\label{defnx}
 n_x:= \sum_{\{(\chi,\iota,q) \in \kC \ : \  \iota(z)=1  \}} q.
\end{equation}
Then fixing a constant $K$ sufficiently large, we set 
\begin{equation}
 p_x:= \min ( \theta_1 e^{-4\beta n_x}, K h).
\end{equation}
The cutoff for low heights is present for technical reasons, and is placed far from the optimal contact density.
As by definition large contours do not cross good cells, the value of  $n_x$ (and thus of $p_x$) does not vary within a good cell. 
Now we define $g(\go)$ as follows 
\begin{equation}
 g(\go):= \prod_{x\in \cG_N}(1+p_x(\xi_x-1))^{\frac{1}{1-\theta}}.
\end{equation}
It corresponds roughly to the partition function of a pinning system for which the $\delta_x$ are IID Bernoulli variables of parameter $p_x$. The power $\frac{1}{1-\theta}$ and the absence of the term $e^h$ are of no importance, 
the proof would work the same if $(1+p_x(\xi_x-1))^{\frac{1}{1-\theta}}$ was replaced by $(1+p_x(e^h\xi_x-1))$. The reason for this choice of $g$  is that it makes the computation slightly simpler.

\medskip

\noindent Now simpler expressions for both terms in \eqref{lolder} can be given.
Let us set
\begin{equation}
 \begin{split}
  \varrho_1(p,\theta)&:= \frac{1-\theta}{\theta}\log \bbE\left[ (1+p(\xi-1))^{\frac{\theta}{1-\theta}}\right],\\
    \varrho_2(p)&:= \log \bbE\left[ (1+p(\xi-1))^{-1}\right],\\
    \varrho_3(p)&:= \log \bbE\left[ (1+p(\xi-1))^{-1}\xi \right]- \log \bbE\left[ (1+p(\xi-1))^{-1}\right].
\end{split}
\end{equation}
We have 
\begin{equation}
\frac{1-\theta}{\theta} \log \bbE\left[g(\go)^{\theta}\right]= \sum_{x\in \cG_N}   \varrho_1(p_x,\theta)
\end{equation}
Concerning the second term in \eqref{lolder}, we have
\begin{equation}\label{lesrhoz}
\bbE\left[ g(\go) e^{\sum_{x\in  \cG_N(\kC)}(\alpha \go_x-\gl(\alpha)+h)\delta_x} \right]
=e^{\sum_{ x\in \cG_N}
 (\varrho_3(p_x)+h) \delta_x+\varrho_2(p_x)}.
\end{equation}
Hence from \eqref{kolder} we have 

\begin{multline}\label{frouzz}
 \frac{1}{\theta}\log 
 \bbE\left[\bE_{N,\gb}\left[ e^{\sum_{x\in  \cG_N(\kC)}(\alpha \go_x-\gl(\alpha)+h)\delta_x} \ | \ \cB_{\kC}\right]^{\theta}
  \right] \\ \le \sum_{x\in\cG_N}   (\varrho_1(p_x,\theta)+ \varrho_2(p_x)) 
 + \log \bE_{N,\beta}\left[ e^{\sum_{ x\in \cG_N}
 (\varrho_3(p_x)+h) \delta_x}  \ | \ \cB_{\kC}\right]
\end{multline}
In order to conclude we first introduce sharp estimates on $\rho_i$ in terms of $p$ and $\theta$.
 These are obtained via tedious but elementary computations which we postpone to Appendix \ref{appcharpesti}.

 \begin{lemma}\label{charpesti}
 There exists a positive constant $C$ (which is allowed to depend on the distribution of $\go$) such that for all  $p, \theta \in (0,1/2)$
 \begin{equation}
 \begin{split}
  \varrho_1(p,\theta)&\le -\frac{p^2}{2}\Var(\xi)+C( p^3+ \theta p^2),\\
 \varrho_2(p)&\le  p^2\Var(\xi)+ Cp^3,\\ 
 \varrho_3(p)&\le -p \Var(\xi)+C p^2.
   \end{split}
 \end{equation}
\end{lemma}
\noindent Now the second ingredient we require is something that could justify an  approximation of the type
$$\bE_{N,\beta}\left[ e^{\sum_{ x\in \cG_N}
 (\varrho_3(p_x)+h) \delta_x}  \ | \ \cB_{\kC}\right] \approx  \prod_{x\in \cG_N} (1+p_x (e^{(\varrho_3(p_x)+h)}-1)),$$
obtained by replacing $\delta_x$ by independent Bernoulli variables of parameter $p_x$.
 Our first step is to factorize the expectation on the l.h.s.\ in order to obtain a product over good cells. It is not possible to obtain an equality since the realization of $\phi$ in the different cells are not independent. However, we can obtain an upper bound by taking for each cell the worse boundary condition.
 
 \medskip
 
Given $\varphi \in \cB_{\kC}$ (a realization of the field for which the set of coarse grained contour is given by $\kC$) and $y$ such that $\cC_y$ is a good cell,  we let $\bP^{\varphi}_y$ denote distribution of $\phi$ conditioned to coincide with $\varphi$ outside of the box $\bar \cC_y$.
More precisely we set 
\begin{equation}
 \bE^{\varphi}_y[f(\phi)]:= 
\bE_{N,\gb}\left[ f(\phi) \ | \ \cB_{\kC}\cap \left\{\phi \restrict_{\bar \cC^{\cc}_y}=\varphi \restrict_{\bar \cC^{\cc}_y}  \right\} \right].
\end{equation}
An immediate induction on $y$ yields, by the domain Markov property 
\begin{equation}\label{wasq}
   \bE_{N,\gb}\left[ e^{\sum_{x\in  \cG_N}[\varrho_3(p_x)+h ]\delta_x} \ | \ \cB_{\kC} \right]\le \prod_{\{ \ y \ : \ \cC_y \text{ is good } \} } \max_{\varphi \in  \cB_{\kC} }
   \bE^{\varphi}_y\left[ e^{\sum_{x\in  \cC_y}[\varrho_3(p_x)+h ]\delta_x} \right]
   \end{equation}
If $h$ is sufficiently small (recall that by definition $p_x\ge c h^{1/3}$, and we assumed $\gep\le 1/20$) we have
$$\varrho_3(p_x)+h\le M^{-2}/2= h^{8\gep}/2.$$ We can thus use the inequality $e^u\le 1+u+u^2$, valid for $u\le 1/2$ and deduce (with a small abuse of notation we write $p_y$ for the value assumed by $p_x$ on  
$\cC_y$) that 
\begin{multline}\label{wqas}
 \log \bE^{\varphi}_y\left[ e^{\sum_{x\in  \cC_y}[\varrho_3(p_y)+h ]\delta_x} \right]
 \le \bE^{\varphi}_y\left[ e^{\sum_{x\in  \cC_y}[\varrho_3(p_y)+h ]\delta_x} \right]-1 \\
 \le\left( [\varrho_3(p_y)+h ]+ M^2[\varrho_3(p_y)+h ]^2 \right)\sum_{x\in  \cC_y}\bE^{\varphi}_y[\delta_x],
\end{multline}
where we used that $\bE^{\varphi}_y[(\sum_{x\in  \cC_y}\delta_x)^2]\le M^2 \sum_{x\in  \cC_y}\bE^{\varphi}_y[\delta_x].$
Back to \eqref{frouzz} combining it with \eqref{wasq}-\eqref{wqas} we obtain that 
\begin{multline}
  \frac{1}{\theta}\log 
 \bbE\left[\bE_{N,\gb}\left[ e^{\sum_{x\in  \cG_N(\kC)}(\alpha \go_x-\gl(\alpha)+h)\delta_x} \ | \ \cB_{\kC}\right]^{\theta}\right]\\ \le 
 \sum_{x\in\cG_N}   (\varrho_1(p_x,\theta)+ \varrho_2(p_x))+\left([\varrho_3(p_x)+h ]+ M^2[\varrho_3(p_x)+h ]^2 \right)\max_{\varphi\in \cB_{\kC}}\bE^{\varphi}_z[\delta_x],
\end{multline}
{\blue where in the above equation we  use the convention announced above \eqref{defnx} that $z=z(x)\in \bbZ^d$  is such that $x\in \cC_z$.}
To conclude it is sufficient to show that for every $x\in \cG_N$, we have 
\begin{multline}\label{presquefini}
 \varrho_1(p_x,\theta)+ \varrho_2(p_x)+\left(\varrho_3(p_x)+h + M^2[\varrho_3(p_x)+h ]^2 \right)\max_{\varphi\in \cB_{\kC}} \bE^{\varphi}_z[\delta_x]\\ \le G_{\beta}(\alpha,h)+h^{2+\gep}.
\end{multline}
We then use an estimate for $\bE^{\varphi}_z[\delta_x]$ that can be obtained as direct consequences of Proposition \ref{rouxrou}  {\blue (recall the definition of $n_x$ \eqref{defnx})}. The proof is postponed to the next subsection.

\begin{lemma}\label{derniezest}
If $x\in \cG_N(\kC)\cap \cC_y$, we have for every realization of $\varphi\in \cB_{\kC}$
\begin{equation}\label{croot}
\left|\bE^{\varphi}_y[\delta_x]- \theta_1 e^{-4\gb n_x}\right|\le C e^{-6\gb n_x}+ e^{-c\sqrt{L}}.
\end{equation}
Furthermore we have 
\begin{equation}\label{root}
\bE^{\varphi}_y[\delta_x]\ge C^{-1} e^{-4\gb n_x}.
\end{equation}

\end{lemma}
Now to prove \eqref{presquefini}, let us first apply the estimates from Lemma \ref{charpesti}. For simplicity we write $p$ for $p_x$. We obtain
 \begin{multline}\label{presquefini2}
 \varrho_1(p,\theta)+ \varrho_2(p)+\left(\varrho_3(p)+h + M^2[\varrho_3(p)+h ]^2 \right)\max_{\varphi\in \cB_{\kC}} \bE^{\varphi}_z[\delta_x]\\
 \le \frac{p^2}{2}\Var(\xi)+(h-p \Var (\xi))\bE^{\varphi}_z[\delta_x]
 + C \left[p^3+\theta p^2+ M^2(h^2+p^2) \bE^{\varphi}_z[\delta_x]\right].
\end{multline}
Now if $\theta_1 e^{-4\beta n}\ge K h$ (and thus $p=Kh$), then  \eqref{root} in Lemma \ref{derniezest} yields  $\bE^{\varphi}_z[\delta_x]\ge C^{-1}Kh $. This is sufficient to ensure that the r.h.s.\ in \eqref{presquefini2} is negative if $K$ is chosen large enough.

\medskip

If $\theta_1 e^{-4\beta n}\le K h$   (and thus $p=\theta_1 e^{-4\beta n}$) then using \eqref{croot} (note here that the correction $e^{-c \sqrt{L}}$ is much smaller than any power of $h$ and thus irrelevant) we obtain that the r.h.s.\ in \eqref{presquefini2} is smaller than (changing the value of the constant $C$ if necessary)
\begin{multline}\label{lazzt}
 h \theta_1 e^{-4\beta n}-\frac{1}{2}\Var (\xi)\theta^2_1 e^{-8\beta n} \\ + C( he^{-6 \beta n}+ e^{-10\beta n}+\theta e^{-8\beta n} +M^2 h^2 e^{-4\beta n}+ M^2 e^{-12 \beta n}+h^{3}).
\end{multline}
The first line in \eqref{lazzt} is clearly smaller than $G_\beta(\alpha,h)$ while
the second line is smaller than $C 'h^{2+\gep}$ which allows to conclude.

\subsection{Proof of Lemma \ref{derniezest}}\label{ccc}

 We obtain the result by considering a stronger conditioning than  
 the realization of $\phi$ outside $\bar \cC_y$.
 We let 
 $\hat \Upsilon^{(y)}(\phi)$ denote the set of cylinders in $\phi$ which intersect $(\bar \cC_y)^{\cc}$, or more precisely
 \begin{equation}
  \hat \Upsilon^{(y)}(\phi):= \{ \hat \gamma \in \Upsilon^{(y)}(\phi), 
\ : \   \tilde \gamma \cap (\bar \cC_y)^{\cc} \ne \emptyset \},
 \end{equation}
where, like in \eqref{cgtrace}, we have identified the geometric contour $\tilde \gamma$ with a subset of $\bbR^2$ (a union of segments).
It is quite simple to check via \eqref{cylinder} that the knowledge of 
  $\hat \Upsilon^{(y)}(\phi)$ is sufficient to reconstruct $\phi \restrict_{\gL_N \setminus \bar \cC_y}$.

 \medskip

 Given $\hat \Gamma$ a collection of cylinders which is such that 
 $\bP\left[ \cB_{\kC} \ | \ \hat \Upsilon^{(y)}(\phi)=\hat \Gamma  \right]>0$ (that is, the coarse grained cylinders corresponding to 
  $\hat \Gamma$ are exactly those in $\kC$), we are going to prove \eqref{croot} and \eqref{root} are valid with $\bP^{\varphi}_y$ replaced by
  \begin{equation}\label{newcondition}
  \bP^{\hat \Gamma}_{y}[\cdot]:= \bP\left[ \cdot \ | \ \cB_{\kC} \cap \{\hat \Upsilon^{(y)}(\phi)=\hat \Gamma \} \right].
 \end{equation}
We want to consider $\bP^{\hat \Gamma}_{y}$, or rather, its restriction to a domain $\gL$ as a measure of the form 
$\bP^n_{\bL,\gL,\beta}$.
 We let $\gL(y, \hat \Gamma)$ denote the subset of $\bar \cC_y\cap \bbZ^2$ obtained by subtracting the interior of contours which intersects the box $\bar \cC_y$
 $$\gL(y, \hat \Gamma):= (\bbZ^2 \cap \bar \cC_y) \setminus  \left( \bigcup_{\{\gamma \in \hat \Gamma \ : \  \tilde \gamma \cap \bar \cC_y\ne \emptyset \}}\bar \gamma \right)$$
 We define similarly $\bL(y,\hat \Gamma)$ to be the set of contours which can appear in  $\Upsilon(\phi) \setminus \Upsilon^{(y)}(\phi)$ when $\Upsilon^{(y)}(\phi)=\gG$. Recall that by convention, $\Upsilon^{(y)}(\phi)$ and $\gG$ correspond to the image of $\hat \Upsilon^{(y)}(\phi)$ and 
 $\hat \gG$ for the canonical projection on the set of cylinders to the set of contours.
 $$\bL(y,\hat \Gamma):= \{ \gamma \in \cC_{\gL} \ : \ |\tilde \gamma| <L \text{ and }  \gamma\cup  \gG \text{ is compatible } \}.$$
 With these definitions, the reader can check that 
 $$\bP^{\hat \Gamma}_{y}[ \phi  \restrict_{ \gL} \in \cdot]= \bP^n_{\gL, \bL, \gb}$$ 
 with $\gL=\gL(y, \hat \Gamma)$ and $\bL=\bL(y,\hat \Gamma)$ defined as above and $n=n_x$.
 As by definition $d(x, \gL^{\cc})\ge L/2$, and thus $d(x,\bL^{\cc})$ (recall \ref{deffdiss})
 is mostly determined by the restriction on the contour length.  We have 
 in particular $d(x,\bL^{\cc})\ge \sqrt{L}/2. $ Hence both inequalities can be deduced from Proposition \ref{rouxrou}.
 \qed

 \medskip

\noindent {\bf Acknowledgements:} The author is expresses his thank to an anonymous referee for a veru detailed reports which helped to improve the quality of the manuscript. He acknowledges funding from a productivity grant from CNPq  
and Young Scientist grant from FAPERJ. 

\appendix

\section{Basic properties for the free energy}
\subsection{Proof of Proposition \ref{freeen}}\label{appintro}
Assuming \eqref{convergence} the statements about convexity and monotonicity of $h\mapsto \tf(\beta,\alpha,h)$ easily follow.
Indeed $\cZ^{n,h,\alpha,\go}_{N,\gb}$ is $\log$-convex in $h$ (that is a function whose $\log$ is a convex function) as a 
sum of $\log$-convex functions $h \mapsto e^{-\gb\cH^{n}_\gL(\phi)+ \sum_{x\in \gL}(\alpha \go- \gl(\alpha)-h)\delta_x}$.
It is also increasing.  
These properties are conserved in the limit so that $\tf(\gb,\alpha,h)$ (provided it exists) is non-decreasing and convex. The identity \eqref{contact} just comes from the fact that 
$$ \partial_h \log \cZ^{n,h,\alpha,\go}_{N,\beta} =\bE^{n,h,\alpha,\go}_{N,\gb}\left[\sum_{x\in \gL_N} \delta_x\right], $$
and convexity allows to interchange limit and derivatives (provided that the limit exists).
 The upper-bound \eqref{annealed} is the classical annealed bound is obtain after passing to the limit in the following inequality
\begin{equation}
\bbE\left[ \log  \cZ^{n,h,\alpha,\go}_{N,\gb} \right]\le 
 \log \bbE\left[ \cZ^{n,h,\alpha,\go}_{N,\gb} \right]=  \log \cZ^{n,h}_{N,\gb},
\end{equation}
where the right hand side denotes the partition function corresponding to $\alpha=0$ (in this case the dependence in $\go$ vanishes).
For the lower-bound, it is sufficient to observe that 

 $$\bbE \left[\log  \cZ^{n,u+\lambda(\beta),\alpha,\go}_{N,\gb}\right]= \bbE\left[ \log \left(
  \sum_{\phi \in \gO_{\gL}}e^{-\gb\cH^{n}_\gL(\phi)+ \sum_{x\in \gL}(\alpha \go +u)\delta_x}\right)\right].$$
is convex in $\alpha$  and minimized when $\alpha=0$ (when all other parameters are fixed). Taking $u= h-\lambda(\beta)$ we obtain  
$$ \bbE \left[\log  \cZ^{n,h,\alpha,\go}_{N,\gb}\right]
\ge \log  \cZ^{n,h-\lambda(\beta)}_{N,\gb}.$$
\medskip

\noindent  Let us thus prove \eqref{convergence}.
A direct consequence of the triangle inequality is that we have for any $\phi \in \gO_{N}$, 
$$|\cH^n_N(\phi)-\cH^m_N(\phi)|\le  4(n-m)N,$$
and thus 
$$|\log  \cZ^{n,h,\alpha,\go}_{N,\gb}-\log  \cZ^{m,h,\alpha,\go}_{N,\gb}|\le 4(n-m)N $$
which implies that  should the limit exist, it does not depend on $n$.
Concerning the existence of the limit in \eqref{convergence},
using the same super-additivity argument as in \eqref{hicup2}, we have, for disjoint $\gL^{(1)}$ and $\gL^{(2)}$ 
\begin{equation}
  \bbE \left[ \log \cZ^{n,h,\alpha,\go}_{\gL^{(1)}\cup \gL^{(2)} ,\gb} \right] \ge   \bbE \left[ \log 
 \cZ^{n,h,\alpha,\go}_{\gL^{(1)} ,\gb}\right] + \bbE\left[  \log \cZ^{n,h,\alpha,\go}_{\gL^{(2)} ,\gb}\right].
\end{equation}
The quantity   $\bbE \left[ \log \cZ^{n,h,\alpha,\go}_{\gL,\gb}\right],$ being invariant under lattice translation (of $\gL$), this is sufficient (e.g.\ following the step of \cite[Exercise 3.3]{cf:Vel}) to entail the convergence of 
$N^{-d}\bbE \left[ \log \cZ^{n,h,\alpha,\go}_{N,\beta}\right].$
To obtain the convergence of the $\log$-partition function without expectation, we need a concentration result. We use a method similar to the one displayed in \cite{cf:CSY} (more specifically for the proof of Proposition 2.5) combining truncation and Azuma's inequality, to prove that there exists 
a positive constant $c$ (depending on $\alpha$ and on the distribution of $\go$) such that for every $u>0$ and $\gL$.
\begin{equation}\label{zuzum}
   \bbP\left[ |\log  \cZ^{n,h,\alpha,\go}_{\gL,\gb}- \bbE\left[ \log  \cZ^{n,h,\alpha,\go}_{\gL,\gb} \right]| \ge \sqrt{|\gL|} u \right] \le 
    3|\gL| \exp(- c u^{2/3}). 
\end{equation}
Let us first work under the assumption that there exists $K<\infty$ such that $\bbP[ |\go_x|\le K]=1$.
Then, considering  $(x_1,\dots,x_{|\gL|})$  an enumeration of $\gL$ we consider the martingale
\begin{equation}
 M_i:= \bbE \left[ \log  \cZ^{n,h,\alpha,\go}_{\gL,\gb} \ | \ (\go_{x_j})_{j\ge i+1}  \right]
\end{equation}
which satisfies, $ |M_{i}-M_{i-1}|\le 2\alpha K$ almost surely for every $i\le |\gL|$.
Thus by Azuma's inequality \cite{cf:A} we have 
\begin{equation}
 \bbP\left[ |M_{|\gL|}-\bbE\left[M_{|\gL|}\right]|\ge  \sqrt{|\gL|}u \right] 
\le 2e^{-\frac{u^2}{8\alpha^2 K^2}}
 \end{equation}
which is stronger than \eqref{zuzum}.
Now, under the weaker assumption \eqref{allorder}
we can use the above for  $\tilde M_i:= \bbE \left[ \log  \cZ^{n,h,\alpha,\tilde \go}_{\gL,\gb} \ | \ (\go_{x_j})_{j\le i}  \right]$ 
where 
$$\tilde \go_x= \ind_{\{|\go_x|\le u^{2/3}\}}\go_x.$$
Using Azuma's inequality for the martingale $\tilde M$ whose increments are smaller than $2 \alpha u^{2/3}$ we obtain that for an appropriate positive constant $c$
\begin{equation}\label{zuzum3}
   \bbP\left[ |\log  \cZ^{n,h,\alpha,\tilde \go}_{\gL,\gb}- \bbE\left[ \log  \cZ^{n,h,\alpha,\tilde \go}_{\gL,\gb} \right]| \ge \sqrt{|\gL|} u/2 \right] \le 
   2 \exp(- c u^{2/3}). 
\end{equation}
 Now we observe that the existence of exponential moments \eqref{allorder} and a union bound implies that $\bbP[ (\go_x)_{x\in \gL} \ne (\tilde \go_x)_{x\in \gL} ]\le 2|\gL|e^{-c u^{2/3}}$ (also for some adequate value of $c$), we obtain thus 
 \begin{equation}\label{zuzum2}
   \bbP\left[ |\log  \cZ^{n,h,\alpha,\go}_{\gL,\gb}- \bbE\left[ \log  \cZ^{n,h,\alpha,\tilde \go}_{\gL,\gb} \right]| \le  \sqrt{|\gL|} u/2 \right] \le 
   2 (|\gL|+1) \exp(- c u^{2/3}). 
\end{equation}
Then we observe that \eqref{zuzum2} implies that 
\begin{multline}\label{diffe}
 \left| \bbE\left[ \log  \cZ^{n,h,\alpha,\go}_{\gL,\gb}\right]- \bbE\left[ \log  \cZ^{n,h,\alpha,\tilde \go}_{\gL,\gb} \right]\right|
 \\ \le   \bbE\left[ \left|\log  \cZ^{n,h,\alpha,\go}_{\gL,\gb}- \bbE\left[ \log  \cZ^{n,h,\alpha,\tilde \go}_{\gL,\gb} \right] \right|\right]
 \le C \sqrt{|\gL|} (\log |\gL|)^{3/2}.
\end{multline}
The reader can then check that after an appropriate change of the constant $c$, \eqref{zuzum} is obtained as a combination of \eqref{zuzum2} and \eqref{diffe}.
Let us finally prove \eqref{labound}. Observe that  we have
\begin{multline}
\frac{\cZ^{n,h,\alpha,\go}_{N,\beta}}{\cZ_{N,\beta}}=\bE^n_{N,\gb}\left[ e^{\sum_{x\in \gL_N}(\alpha \go- \gl(\alpha)+h)\delta_x}\right] 
\ge \bP^n_{N,\beta}[\forall x\in \gL_N, \ \phi(x)\ge  1 ]\\
=  \bP_{N,\beta}[\forall x\in \gL_N, \ \phi(x)\ge  1-n].
\end{multline}
Hence to conclude one must show that for any $\gep$, we can find $n$ such that for $N$ sufficiently large we have
\begin{equation}
 \bP_{N,\beta}[\forall x\in \gL_N, \ \phi(x)\ge  1-n ]\ge e^{-\gep N^d}.
\end{equation}
Let $N_0$ be such that 
$$ \frac{1}{N^d_0}\log \cZ_{N_0,\beta} \ge \tf(\beta)-(\gep/2) \quad \text{ and }  \quad N_0 \ge \frac{d \tf(\beta)}{\gep} $$
Note that we only need to prove this for a sequence of values of $N$ going to infinity so we might assume that $N$ is of the form $kN_0+k-1$ for some $k\ge 1$.
We consider $\bbH(N,N_0)$ to be a ``grid'' which splits $\gL_N$ into cells of side-length $N_0$,  that is the set of vertex in $\gL_N$ for which at least one coordinate is a multiple of $N_0+1$.
\begin{equation}
 \bbH(N,N_0):= \{ (x_1,\dots,x_d)\in \gL_N \ : \ \exists i\in \lint 1,d\rint ,   x_i/(N_0+1)\in \bbZ \}
\end{equation}
Now we can bound from below the probability we want to estimate by
\begin{multline}
    \bP_{N,\beta}\left[ \{ \forall x\in \gL_N, \ \phi(x)\ge  1-n\ \} \cap \{ \forall x\in \bbH, \phi(x)=0 \}  \right]\\ = \bP_{N,\beta}[ \forall x\in \gL_N, \ \phi(x)\ge  1-n \ | \  \forall x\in \bbH, \phi(x)=0 ] \bP_{N,\beta}[  \forall x\in \bbH, \phi(x)=0 ],
\end{multline}
and estimate separately each term of the product.
Using the spatial Markov property for the SOS model, we have
\begin{multline}
 \bP_{N,\beta}[ \forall x\in \gL_N, \ \phi(x)\ge  1-n \ | \  \forall x\in \bbH, \phi(x)=0 ]\\
 = \bP_{N_0,\beta}\left[ \forall x\in \gL_{N_0}, \ \phi(x)\ge  1-n\right]^{k^d} \ge e^{- N^d (\gep/3)}.
\end{multline}
where the last inequality can be made valid by choosing $n$ sufficiently large. Now just from the definition of the partition functions we have
(recall that $\cZ_{N,\beta}\le e^{N^d \tf(\beta)}$ for every $N$)
\begin{equation}
 \bP_{N,\beta}[  \forall x\in \bbH, \phi(x)=0 ] = \frac{(\cZ_{N_0,\beta})^{k^d}}{\cZ_{N,\beta}} \ge \exp\left( k^d N^d_0 (\tf(\beta)-(\gep/2))- N^d \tf(\beta)  \right) . 
 \end{equation}
 The quantity in the exponential is larger than $-\gep N^d$.

\qed

\subsection{Proof of Proposition \ref{firstorder}}
We have 
$$\bar \tf_{\gb}(h)=\lim_{N\to \infty}\frac{1}{N^d}\log \bE_{N,\gb}\left[ e^{h\sum_{x\in \lint 1, N\rint^d} \delta_x}\right].$$
As a consequence of Jensen's inequality, when $h\ge 0$, the r.h.s.\ is larger than (the existence of the limit and its value are consequences of \eqref{decayz}) 
\begin{equation}
h \lim_{N\to \infty}\frac{1}{N^d}\bE_{N,\gb}\left[ \sum_{x\in \lint 1, N\rint^d} \delta_x\right]=h \bP_{\beta}(\phi({\bf 0})=0).
\end{equation}
When $h\le 0$ we have $\bE_{N,\gb}\left[ e^{h\sum_{x\in \lint 1, N\rint^d} \delta_x}\right]\le 1$ which implies   $\bar \tf_{\gb}(h)\le 0$. On the other hand as a consequence of \eqref{labound}) we have  $\bar \tf_{\gb}(h)\ge 0$ and hence  $\bar \tf_{\gb}(h)=0$ for $h\le 0$.
The result follows by convexity for some constant $c_{\beta}$ satisfying $\bP_{\beta}(\phi({\bf 0})=0)\le c_{\beta}< 1$.
\qed 
\section{Proof of Proposition \ref{rouxrou}}\label{lazt}
Except for \eqref{lotzof},  the estimates we need to prove are very similar to the ones found in \cite{cf:part1} and \cite{cf:HL}, with only minor modifications needed in the proof. We include the details for  \eqref{czups}, \eqref{zup2s} and \eqref{zup3s} for the sake of completeness.

\subsection{Proof of \eqref{czups}}

A proof using cluster expansion is given in \cite[Proposition 3.10]{cf:HL} (for a constant $C=2$). Let us present a simpler proof which only use a variant of Peierls argument.

\medskip

Recall that $\gamma_x$ denote the positive contour of length $4$ such that $\bar \gamma_x=\{x\}$. One possibility to have $\phi(x)=n$ is that $\gamma_x$ is the only contour in $\Upsilon (\phi)$ surrounding $x$ and it has intensity $n$.
Set 
$$\bL^{0}_x:=\{ \gamma \in \bL\setminus \{\gamma_x\} \ : \ \gamma \text{ is incompatible with } \gamma_x \text{ or } x\in \bar \gamma \}.$$
We have
\begin{multline}
 \bP_{\bL,\gL,\beta}[ \phi(x)= n]\ge  \bP_{\bL,\gL,\beta}[ \gamma_x\in \Upsilon(\phi)\ ;  \ k(\gamma_x)=n \  ;  \Upsilon(\phi) \cap \bL^{0}_x=\emptyset \ ]\\
 = \bP_{\bL,\gL,\beta}\left[ \gamma_x\in \Upsilon(\phi)\ ;  \ k(\gamma_x)=n \   \ | \   \Upsilon(\phi) \cap \bL^{0}_x=\emptyset \ \right]\bP_{\bL,\gL,\beta} \left[\Upsilon(\phi) \cap \bL^{0}_x=\emptyset \ \right].
\end{multline}
Using Lemma \ref{restrict}, because there is no compatibility constraint 
present after the conditioning, $\gamma_x$ is present with probability 
$e^{-4\beta}$ and $k(\gamma_x)$ is a geometric random variable, hence 
\begin{equation}
 \bP_{\bL,\gL,\beta}\left[ \gamma_x\in \Upsilon(\phi)\ ;  \ k(\gamma_x)=n \   \ | \   \Upsilon(\phi) \cap \bL^{0}_x=\emptyset \ \right]= (1-e^{-4\beta}) e^{-4\beta n}.
\end{equation}
Now $\Upsilon(\phi) \cap \bL^{0}_x=\emptyset$ is a decreasing event. Hence using Lemma \ref{restrict} (the stochastic domination by a product measure) we have
\begin{equation}
 \bP_{\bL,\gL,\beta} \left[\Upsilon(\phi) \cap \bL^{0}_x=\emptyset \ \right]\ge \prod_{\gamma \in \bL^{0}_x}(1-e^{-\beta |\tilde \gamma|}).
\end{equation}
The r.h.s.\ is bounded from below uniformly in $\gL$ and $\bL\in \gL$ provided that 
$ \sum_{\{\gamma\in \cC : {\bf 0} \in \bar\gamma \}}e^{-\beta |\tilde \gamma|}<\infty$ which is valid for large enough $\beta$.
\qed

\subsection{Proof of \eqref{zup2s}}

If $\phi(x),\phi(y)\ge n$, then there must exists a set of positive contours enclosing $x$ whose intensities sum up to something larger than $n$. The same must be satisfied for $y$. 
We let $p$ be the sum of intensities of contours which enclose both $x$ and $y$. Letting $\Upsilon_+(\phi)$ denote the set of positive contours in $\phi$ we introduce the events
\begin{equation}
 \begin{split}
 A_p&:= \left\{   \exists m, \exists (\gamma_i)_{i=1}^m \in  (\Upsilon_+(\phi))^m, \{x,y\} \subset \bar \gamma_1 \subset \dots \subset \bar \gamma_m,  \text{ and }    \sum_{i=1}^m k(\gamma_i)= p \right\}\\
B^{(x)}_{q}&:=  \left\{  \exists m, \exists (\gamma_i)_{i=1}^m \in  (\Upsilon_+(\phi))^m, \{x\} \subset \bar \gamma_1 \subset \dots \subset \bar \gamma_m, \ y\notin  \bar \gamma_m  , \ \sum_{i=1}^m k(\gamma_i)\ge q \right\},\\
B^{(y)}_{q}&:=  \left\{ \exists m, \exists (\gamma_i)_{i=1}^m \in  (\Upsilon_+(\phi))^m, \{y\} \subset \bar \gamma_1 \subset \dots \subset \bar \gamma_m, \ x\notin  \bar \gamma_m  ,  \  \sum_{i=1}^m k(\gamma_i)\ge q \right\},
 \end{split}
\end{equation}
with the convention that $A_p$, $B^{(x)}_{q}$ and $B^{(y)}_{q}$ are equal to $\gO_{\gL}$ for $p,q\le 0$. We also assume that their is not repetition in our sequences of contour, that is,  the inclusion $\bar \gamma_i \subset \bar \gamma_{i+1}$ is strict.
We have by a union bound

\begin{equation}
\bP_{\bL,\gL,\beta}(\phi(x),\phi(y)\ge n)\le \sum_{p=0}^{\infty}
\bP_{\bL,\gL,\beta}\left[A_p\cap B^{(x)}_{n-p} \cap B^{(y)}_{n-p}\right]
\end{equation}
Now as $A_p\cap B^{(x)}_{n-p} \cap B^{(y)}_{n-p}$ is an increasing event, we can use Lemma \ref{restrict} and independence of the events under $Q_{\gL,\beta}$  to obtain that the probability we wish to bound is smaller than   
\begin{equation}
 \sum_{p=0}^{\infty}
\bQ_{\gL,\beta}(A_p) \bQ_{\gL,\beta}(B^{(x)}_{n-p})\bQ_{\gL,\beta}(B^{(y)}_{n-p})
\end{equation}
Hence we can conclude if we can show that for all $p$ and $q$ we have 
\begin{equation}\label{laffair}
 \bQ_{\gL,\beta}(A_p)\le C e^{-6\beta p} \quad  \text{ and } \bQ_{\gL,\beta}(B^{(x)}_{q}), \bQ_{\gL,\beta}(B^{(x)}_{q})\le C e^{-4\beta q}.
\end{equation}
Indeed this would  directly imply
\begin{equation}
 \bP_{\bL,\gL,\beta}(\phi(x),\phi(y)\ge n) \le \sum_{p=0}^n e^{ 8\beta n - 2 \beta p}+ \sum_{p>n} e^{-6\beta n}.
\end{equation}
We focus on the first probability in \eqref{laffair} which is more delicate to control (the other one in controlled in a completely analogous fashion). Using union bound and ignoring the constraint that the contours are in $\cC_{\gL}$ as well as a factor $(1-e^{-\beta |\tilde \gamma_i|}) $ for each contour, we have 
\begin{equation}
  \bQ_{\gL,\beta}(A_p) \le \sum_{1\le m\le p} \  \ \sum_{\{x,y\}\subset \bar \gamma_1\subset \dots \subset \bar \gamma_m } \sum_{\{ (k_i) \ : \ \sum_{i=1}^m k_i=p \}} e^{- \beta\sum_{i=1}^m k_i |\tilde \gamma_i|}
\end{equation}
Note that we necessarily have $|\tilde \gamma_1|\ge 6$ and $|\tilde \gamma_2|\ge 8$
\begin{equation}
 \sum_{i=1}^m k_i |\tilde \gamma_i| \ge 6(k_1-1) + 
8 \sum_{i=2}^m (k_i-1) + \sum_{i=1}^m |\tilde \gamma_i|=
8(p-m)-2(k_1-1)+  \sum_{i=1}^m |\tilde \gamma_i|.
\end{equation}
With this observation we can bound the sum over $(k_i)$ as we have
(the first inequality results from the fact that  $2^{p-k_1-1}$ is a bound for the number of partitions of an inteval of length $p-k_1$)
\begin{equation}
 \sum_{\{ (k_i) \ : \ \sum_{i=1}^m k_i=p \}} e^{2\beta (k_1-1)}
 \le \sum_{k_1=1}^{p-m+1} e^{2\beta k_1} 2^{p-k_1-1}\le e^{2\beta (p-m)}(1+ 2 e^{-2\beta}).
\end{equation}
We obtain (with a constant allowed to depend on $\beta$)
\begin{equation}
   \bQ_{\gL,\beta}(A_p) \le   C  e^{-6\beta p} \sum_{1\le m\le p}   e^{6\beta m}  \ \sum_{\{x,y\}\subset \bar \gamma_1\subset \dots \subset \bar \gamma_m } e^{- \beta\sum_{i=1}^m  |\tilde \gamma_i|}
\end{equation}
Now to estimate the sum over contours, note that we have 
$|\bar  \gamma_i|\ge i+1$. Let us fix $K$ (depending only on $\beta$) which is such that 
\begin{equation}
 \sum_{\{ \gamma \ : \ x\in \bar\gamma,  \ |\tilde \gamma|\ge K+1 \}} e^{-\beta  |\tilde \gamma|}\le e^{-8\beta}.
\end{equation}
Then relaxing the restriction $\{x,y\}\subset \bar \gamma_1\subset \dots\subset \bar \gamma_m$ and replacing it by 
$x\in \bar\gamma,  \ |\tilde \gamma|\ge K$ for $i\ge K+1$ and simply 
$x\in \bar\gamma$ when $i\le K$ we obtain that 
\begin{equation}
 \sum_{\{x,y\}\subset \bar \gamma_1\subset \dots \subset \bar \gamma_m } e^{- \beta\sum_{i=1}^m  |\tilde \gamma_i|}
 \le \left(  \sum_{\bar \gamma \ni x} e^{-\beta |\tilde \gamma|}\right)^{m\wedge K}
  e^{8\beta (m-K)_+}\le C e^{-8\beta m}.
\end{equation}
This is sufficient to conclude that    $\bQ_{\gL,\beta}(A_p)\le C e^{-6\beta p}$.
\qed

\subsection{Proof of \eqref{zup3s}}

Using the triangle inequality we have 
\begin{multline}
 \left|\bP_{\bL,\gL,\beta}[ \phi(x)= n] - \theta_1 e^{-4\beta n}\right|
 \\ \le  \left|\bP_{\bL,\gL,\beta}[ \phi(x)= n] - \bP_{\beta}[\phi(x)= n]\right|+\left|\bP_{\gL,\beta}[ \phi(x)= n] - \theta_1 e^{-4 \beta n}\right|
\end{multline}
Now using the fact that $\{\phi(x)=n\}$ is a \text{local event} in the acception given in  \cite[Proposition 3.5]{cf:HL} we have using this proposition 
\begin{equation}
 \left|\bP_{\bL,\gL,\beta}[ \phi(x)= n] - \bP_{\beta}[\phi(x)= n]\right|
 \le e^{-(\beta/100)d(x,\bL^{\cc})}.
\end{equation}
Now concerning the second term, it is controlled by \cite[Proposition 4.6]{cf:part1}, we have 
\begin{equation}
 \left|\bP_{\gL,\beta}[ \phi(x)= n] - \theta_1 e^{-4 \beta n}\right|\le C e^{-6\beta n}.
\end{equation}
Note that the result in \cite{cf:part1} concerns $\bP_{\gL,\beta}[ \phi(x)\ge  n]$ but it ends up being equivalent to the above at the cost of a change of the constant $C$.
\qed
\subsection{Proof of \eqref{lotzof}}

 For the event to be realized, we need to have a family of positive cylinders 
 $(\gamma_j,k_j)_{j=1}^m$ in $\hat \Upsilon(\gamma)$  such that 
  
 \begin{equation}\label{inc}
 \begin{split}
\forall j \in \lint 1, m \rint,\quad  &\exists i \in \lint 1,\kappa\rint, \quad   x_i \in \bar \gamma_j,\\
 \forall i\in \lint 1,\kappa\rint, \quad  &\sum_{j=1}^{m} k_j \ind_{\{x_i\in \bar \gamma_j\}} \ge n.
 \end{split}
 \end{equation}
 Using Lemma \ref{restrict}
 we recall that 
 $$\bP_{\bL,\gL,\beta}\left[ \{(\gamma_j,k_j)\}_{j=1}^m \subset \hat \Upsilon(\phi)\right]\le \prod_{j=1}^m (1-e^{-\beta |\tilde \gamma_i|})e^{-k_i \beta|\tilde \gamma_i|} $$ 
 Thus, using a union bound, if we  let $\hat \cC^+_{{\bf x},n}$ be the set of  compatible families of cylinders which satisfies \eqref{inc}, it is sufficient to prove that 
\begin{equation}\label{letruc}
 \sum_{ \hat \gG \in \hat \cC^+_{{\bf x},n}} \prod_{\hat \gamma \in  \hat \gG }e^{-\gb k(\hat \gamma) |\tilde \gamma|}\le  C_{\kappa,\gep} e^{-3\gb \sqrt{\kappa} n}.
\end{equation}
Let us prove \eqref{letruc} by induction on $\kappa$.
Our statement includes the case $n =0$ (for which in particular one is allowed to include the empty cylinder family).
Note that the case $\kappa=1$ is a consequence of the proof of 
\eqref{zups}. To perform the induction step we split our collection of  $\hat \gG$ between cylinders which enclose all the $x_i$s and cylinders  {\blue which enclose a strict (non-empty) subset of $\{x_i\}_{i\in \lint 1,\kappa\rint}$}. We let $\hat \cC^{0}_{{\bf x},m}$  denote  families of compatible cylinders whose contours enclose all the $(x_i)_{i=1}^{\kappa}$ and whose sum of intensities is equal to $m$. Using the notation ${\bf y} \subset {\bf x}$ to say that ${\bf y}$ is a sequence obtained by deleting some of the $x_i$s we obtain
\begin{multline}
  \sum_{ \hat \gG \in \hat \cC^+_{{\bf x},n}} \prod_{\hat \gamma \in  \hat \gG }e^{-\gb k(\hat \gamma) |\tilde \gamma|} 
   \le  \sum_{m=0}^\infty
  \sum_{ \hat \gG_1 \in \hat \cC^0_{{\bf x},m}} 
  \prod_{\hat \gamma \in  \hat \gG_1 } e^{-\gb k(\hat \gamma_1) |\tilde \gamma_1|} \\
  \times \sumtwo{{\bf  y} \subset {\bf x}}{ {\bf y} \notin \{ \bf x, \emptyset\}}\left(\sum_{ \hat \gG_2 \in \hat \cC^+_{{\bf y},n-m}} \prod_{\hat \gamma_2 \in  \hat \gG _2} e^{-\gb k(\hat \gamma_2) |\tilde \gamma_2|}\right) \left(\sum_{ \hat \gG_3 \in \hat \cC^+_{{\bf x \setminus y},n-m}} \prod_{\hat \gamma_3 \in  \hat \gG_3 } e^{-\gb k(\hat \gamma_3) |\tilde \gamma_3|}\right)
  \end{multline}
If we let $|{\bf y}|$ denote the number of terms in the sequence ${\bf y}$. The induction hypothesis implies that 
\begin{multline}
 \left(\sum_{ \hat \gG_2 \in \hat \cC^+_{{\bf y},n-m}} \prod_{\hat \gamma_2 \in  \hat \gG _2} e^{-\gb k(\hat \gamma_2) |\tilde \gamma_2|}\right) \left(\sum_{ \hat \gG_3 \in \hat \cC^+_{{\bf x \setminus y},n-m}} \prod_{\hat \gamma_3 \in  \hat \gG_3 } e^{-\gb k(\hat \gamma_3) |\tilde \gamma_3|}\right)\\
 \le C_{|{\bf y}|}C_{\kappa-|{\bf y}|} e^{-3\beta (\sqrt{\kappa-|y|}+ \sqrt{|y|})(n-m)_+}
 \le C_{|y|}C_{\kappa-|{\bf y}|}e^{-3\beta (\sqrt{\kappa-1}+1)(n-m)_+}
\end{multline}
Summing over all ${\bf y}$ only factors a constant that depends only on $\kappa$, hence to conclude 
 we only need to check that for some constant depending on $\kappa$ we have for every $m\ge 0$
\begin{equation}
 \sum_{ \hat \gG \in \hat \cC^0_{{\bf x},m}} e^{-\gb  \sum_{\gamma \in \hat \gG}k(\hat \gamma) |\tilde \gamma|}\le  C(\kappa)e^{-3\gb  \sqrt{\kappa} m}. 
\end{equation}
From now on, we present an element $\hat \gG\in  \hat \cC^0_{{\bf x},m}$ as a sequence $((\gamma_1,k_1), \dots, (\gamma_l,k_l))$ where (identifying ${\bf x}$ with a set of points) ${\bf x} \subset \bar \gamma_1\subset \dots \subset \bar \gamma_l$ (and $\sum k_i=m$). 
Note that by isoperimetric inequality, all contours have length at least 
$4\sqrt{\kappa}$.
Hence  we have 
\begin{equation}
 \sum_{\hat \gamma \in \hat \gG}k(\hat \gamma) |\tilde \gamma|
 \ge  \sum_{\gamma \in \hat \gG}|\tilde \gamma|
 +(m-l)4\sqrt{\kappa}.
\end{equation}
Now let us choose $K$ (depending only on $\kappa$ and $\gb$) such that 
$$\sum_{\{ \gamma \  : \  {\bf 0}\in \bar\gamma \text{ and } |\bar\gamma|\ge K+1\}} e^{-\gb |\tilde \gamma|}\le e^{-4\gb(\sqrt{\kappa}+1)}.$$
Proceeding as in the proof of \eqref{zup2s} we observe that as our sequence of contours is nested we have  $|\bar \gamma_i|\ge i$. This entails that
\begin{multline}
\sum_{{\bf x} \subset \gamma_1 \subset \dots \subset \gamma_l } \prod_{i=1}^l e^{-\gb k_i|\tilde \gamma_i|}\\
\le  
 C(\kappa,\gb) e^{-\gb(m-l)4\sqrt{\kappa}}
  \left(\sum_{\gamma\supset {\bf x}}e^{-\gb |\tilde \gamma|} \right)^{K\wedge l} \left( \sum_{\{ \gamma \ : \bar \gamma\supset {\bf x} , |\tilde \gamma|\ge K+1  \}}e^{-\gb |\tilde \gamma|}\right)^{(l-K)_+}\\
  \le C'(\kappa, \gb) 
  e^{-\gb[ 4m \sqrt{\kappa}+l]}
\end{multline}
Hence summing over $l$ and ${\bf k}$, using the fact that 
$$   \#\{{\bf k} \ : \ \sum_{i=1}^l k_i =m  \}=  \binom{m-1}{l-1}.$$
we obtain
\begin{multline}
\sum_{l=1}^{m} \sum_{ \{{\bf k} \ : \ \sum_{i=1}^l k_i =m  \}}\sum_{{\bf x} \subset \gamma_1 \subset \dots \subset \gamma_l } e^{-\gb k_i|\tilde \gamma_i|}\\
 \le  C'(\kappa, \gb)   e^{-4\gb m \sqrt{\kappa}}
 \sum_{l=1}^{m} e^{-4\gb l} \binom{m-1}{l-1}\le C''(\kappa, \gb)  e^{-3\gb m \sqrt{\kappa}}
\end{multline}
and that concludes the proof.

\section{Proof of Lemma \ref{charpesti}}\label{appcharpesti}

Let us start with $\varrho_2$.
Note that when $x\ge -1/2$ we have by Taylor's formula  
$$\frac{1}{1+x}\le 1-x+x^2+ 8|x|^3.$$
Now as $p\ge 1/2$ by assumption we obtain 
\begin{multline}
 \log \bbE[ (1+p(\xi-1))^{-1}] = \bbE[ (1+p(\xi-1))^{-1}] -1 \\
 \le - p\bbE[(\xi-1)]+ p^2\bbE[(\xi-1)^2] + 8 p^3
 \bbE[|\xi-1|^3]= p^2 \Var(\xi)+  8 p^3\bbE[|\xi-1|^3].
\end{multline}
Concerning $\varrho_3$, using the expansion with one fewer term $\frac{1}{1+x}\le 1-x+4x^2$, valid for the same values of $x$ we have
\begin{multline}
 \bbE[ (1+p(\xi-1))^{-1} \xi]
 \le \bbE[\xi]- p\bbE[(\xi-1)\xi]+ 4p^2\bbE[(\xi-1)^2\xi]\\= 1+p \Var (\xi)+4 p^2 \bbE[(\xi-1)^2\xi].
\end{multline}
Using again that $\log \le u-1,$ and the result concerning $\varrho_2$, it is sufficient to conclude.
Now concerning $\varrho_1$ we use the inequality (obtained from Taylor expansion, valid for any $x\ge -1/2$ and $\gamma\in (0,1)$. 

$$ (1+x)^{\gamma} \le 1+\gamma x -\frac{\gamma(1-\gamma)}{2} x^2 +  4\gamma   |x|^3$$
We obtain (assuming $\theta\le 1/2$)

\begin{equation}
  \bbE[ (1+p(\xi-1))^{\frac{\theta}{1-\theta}}]-1\le  1- p^2\frac{\theta(1-2\theta)}{(1-\theta)^2} \Var(\xi) +  \frac{4\theta}{1-\theta} p^3 \bbE[|\xi-1|^3]. 
\end{equation}
This yields in turn 
\begin{equation}
   \frac{1-\theta }{\theta}\log \bbE[ (1+p(\xi-1))^{\frac{\theta}{1-\theta}}]\le -\frac{p^2}{2} \left( 1- \frac{\theta}{1-\theta}\right)+ 4p^3 \bbE[|\xi-1|^3],  
\end{equation}
and concludes our argument.

\qed

\end{document}